\documentclass{amsart}

\usepackage{csquotes}

\usepackage[hidelinks]{hyperref}

\usepackage[margin=1.0in]{geometry}
\usepackage{amsmath}
\usepackage{amssymb}
\usepackage{amsthm}
\usepackage{enumerate}
\usepackage{xcolor}
\usepackage{cleveref}
\usepackage{tikz}
\usepackage[T1]{fontenc}

\newtheorem{thm}{Theorem}[section]
\newtheorem{lem}[thm]{Lemma}

\theoremstyle{definition}
\newtheorem{dfn}[thm]{Definition}

\newtheorem*{prb*}{Problem}
\newtheorem*{thm*}{Theorem}
\newtheorem{rmk}[thm]{Remark}
\newtheorem{question}[thm]{Question}

\providecommand{\e}{\varepsilon}
\providecommand{\emp}{\varnothing}

\providecommand{\x}{\xi}

\providecommand{\R}{\mathbb{R}}

\providecommand{\Z}{\mathbb{Z}}
\providecommand{\N}{\mathbb{N}}

\newcommand{\ii}{{\bf i}}
\newcommand{\jj}{{\bf j}}

\newcommand{\cC}{\mathcal{C}}
\newcommand{\cD}{\mathcal{D}}

\newcommand{\cT}{\mathcal{T}}
\newcommand{\cV}{\mathcal{V}}

\newcommand{\LINE}{\mathsf{LINE}}

\newcommand{\VB}{\mathsf{VB}}
\newcommand{\IVB}{\mathsf{IterVB}}
\newcommand{\DIVIDE}{\mathsf{DIVIDE}}
\newcommand{\ROTATE}{\mathsf{ROTATE}}

\newcommand{\nbhd}[2]{\operatorname{nbhd}_{#2}(#1)}
\newcommand{\dom}{\mathrm{dom }\thinspace }
\newcommand{\intdom}{\mathrm{int\thinspace dom }\thinspace }

\newcommand{\sma}{\mathrm{small}}
\newcommand{\cov}{\mathrm{cover}}

\newcommand{\newA}{W}

\numberwithin{figure}{section}
\numberwithin{equation}{section}

\usepackage{enumitem}
\setlist[enumerate]{topsep=1ex,itemsep=0ex,partopsep=1ex,parsep=1ex}

\title[Prescribed projections and efficient coverings]{\centering{Prescribed projections and efficient coverings\\ by curves in the plane}}

\author{Alan Chang}
\address{Department of Mathematics, Washington University in St. Louis, St. Louis, MO}
\email{alanchang@math.wustl.edu}

\author[Alex McDonald]{Alex McDonald }
\address{Department of Mathematics, Kennesaw State University, Marietta, GA}
\email{amcdon79@kennesaw.edu}

\author[Krystal Taylor]{Krystal Taylor}
\address{Department of Mathematics, The Ohio State University, Columbus, OH}
\email{taylor.2952@osu.edu}

\thanks{The first author is supported by the National Science Foundation under Grant No.~DMS-2247233. During part of this project, the second author was in residence at the Institute for Computational and Experimental Research in Mathematics in Providence, RI, during the Harmonic Analysis and Convexity program, which was supported by the NSF Grant No.~DMS-1929284. The third author is supported in part by the Simons Foundation Grant MPS-TSM-00007969.}

\begin{document}
\maketitle
\begin{abstract}
Measure zero sets containing intricate structure are foundational in geometric measure theory. These include Besicovitch sets in the plane (also known as Kakeya sets), which are measure zero sets containing a line in every direction. Closely related to the concept of Besicovitch sets is Davies' efficient covering theorem, which states that an arbitrary measurable set $\newA$ in the plane can be covered by full lines so that the union of the lines has the same measure as $\newA$. This result has an interesting dual formulation in the form of a prescribed projection theorem. 
In this paper, we develop each of these results in a nonlinear setting and consider some applications. 
In particular, given a measurable set $\newA$ and a curve $\Gamma=\{(t,f(t)): t\in [0,1]\}$, where $f$ is $C^1$ with strictly monotone derivative, 
we show that $\newA$ can be covered by translations of $\Gamma$ in such a way that the union of the translated curves has the same measure as $\newA$. 
This is achieved by proving an equivalent prescribed generalized projection result, which relies on 
a Venetian blind construction.  
\end{abstract}

\section{Introduction}
R. O. Davies \cite[Theorem 1]{Davies52} proved the following remarkable covering result.
\begin{displayquote}
\textit{Given any plane set $\newA$ of finite Lebesgue measure, there exists a set $L$ consisting of a union of straight lines such that $\newA \subset L$ and  $|\newA| =  |L|$.
}
\end{displayquote}
We refer to this result as {\it Davies' efficient covering} theorem.  
Here and throughout, the notation $|E|$ will be used to denote the Lebesgue measure of $E$ in the appropriate dimension.
An extension of Davies where Lebesgue measure is replaced by $\sigma$-finite Borel measures was proved by Cs\"ornyei \cite{C01}, but we focus on the Lebesgue measure.
\\

Davies' result fits into a broader context of the study of linear accessibility and the related Kakeya conjecture. 
A {\it Kakeya set}, also known as a {\it Besicovitch set}, is a Borel subset of $\R^n$ of measure zero that contains a unit line segment in every direction.  
Among the most important conjectures in the intersection of harmonic analysis and geometric measure theory, the Kakeya conjecture asserts that every Besicovitch set in $\R^n$ has Hausdorff dimension $n$.  This conjecture was proved by Davies \cite{DaviesKakeya} when $n=2$ and, in a very recent breakthrough, by Wang and Zahl \cite{WZ25} when $n=3$.  It is open for $n\geq 4$  \cite[Conjecture 11.4]{Mat15}. 
It continues to be of current interest, \cite{HRZ, KZ19, RenWang, WangZahl, Wisewell05, Zahl2021New, Z23} and has nonlinear analogues \cite{CC, OF23, Keleti17}. 
Various methods of constructing Kakeya sets are discussed in \cite[Section 11.6]{Mat15}, \cite[Chapter 7]{Fal86_book}, and  \cite[Chapter 10]{W03}. 
\\

A related thin set is that of a {\it Nikodym set}.
A Borel set $N\subset \R^n$ is a Nikodym set if its Lebesgue measure is zero and for every $x\in \R^n$, there is a line $\ell$ through $x$ so that $N\cap \ell$ contains a unit line segment. 
So, while a Besicovitch set is a set of measure zero containing a line segment in every direction, 
a Nikodym set is a set of measure zero containing a positive proportion of a line through each point.  
Davies simplified and improved Nikodym's original construction by showing that it is possible to construct a full measure set in the plane so that each point in the set is accessible by $2^{\aleph_0}$-many lines \cite{Davies52}. See also \cite{Mat15}, and see \cite[Section 5]{CCHK} for a different construction via Baire category. Curved Nikodym sets are studied in \cite{CC, cdk}.
The Nikodym conjecture, \cite[Conjecture 11.10]{Mat15}, 
asserts that 
every Nikodym set in $\R^n$ has Hausdorff dimension $n$. 
Every Nikodym set can be transformed into a Kakeya set of the same dimension by a projective transformation \cite[Theorem 11.11]{Mat15}; this means that the Kakeya conjecture implies the Nikodym conjecture. \\

Davies' original proof of his efficient covering theorem as well as some constructions of Kakeya and Nikodym sets rely on so-called Perron trees, where one starts with a triangle, partitions the base to form smaller triangles, and rearranges the pieces; see \cite{Davies52} and \cite[Theorems 7.2, 7.3]{Fal86_book}.  An alternative way to prove Davies' efficient covering theorem is via projections.  
The following is a dual formulation of Davies' theorem \cite[Theorem 7.11]{Fal86_book}:
\begin{displayquote}
\textit{(Prescribed Projection Theorem:)
Let $\newA_\theta$ be a subset of $\R$ for each $\theta \in [0,\pi)$, 
and suppose that the set  $\{(\theta, y) : y\in \newA_\theta \}$ is plane Lebesgue measurable. 
Then, there exists a set $E\subset \R^2$ such that for almost all $\theta$ both of the following hold: 
\begin{enumerate}
\item[(i)]
$\newA_\theta \subset \pi_\theta (E)$;
\item[(ii)]
$|\newA_\theta|  = |\pi_\theta (E)| $, 
\end{enumerate}
where
$\pi_\theta: \R^2 \rightarrow \R$
 is defined by $\pi_\theta(x) = x\cdot e_\theta$, and $e_\theta$ is the unit vector in direction $\theta$. 
}
\end{displayquote}

The equivalence of Davies' efficient covering theorem and the prescribed projection is demonstrated below in Appendix \ref{duality} and uses the duality between points and lines. 
In fact, much inspiration for this paper comes from Falconer's \cite{Fal86_paper}, where he proves a higher dimensional analogue of the prescribed projection theorem; when $d=3$, this is known as the digital sundial \cite{Fal03}. 
\\

In this paper, we investigate analogues of Davies' efficient covering theorem and the prescribed projection theorem in the setting where lines are replaced by 
translates of a fixed curve.  
In particular, we consider the following question: 
\textit{
for what curves $\Gamma$ does the family of translates $\{x+\Gamma: x\in \R^2\}$ admit efficient coverings of sets in the plane?  
}

Our main theorem is as follows. 
\begin{thm}[Efficient covering by curves]
\label{maintheorem}
Let $\newA\subset \R^2$ be measurable.
Let $f:[0,1]\to\R$ be a $C^1$ function with a strictly monotone derivative, and let $\Gamma \subset \R^2$ be the graph of $f$.  Then there exists a set $E\subset \R^2$ such that 
\begin{enumerate}[label=(\alph*)]
\item $E+\Gamma\supset \newA$,
\item $|(E+\Gamma)\setminus \newA|=0$.
\end{enumerate}
\end{thm}

\begin{rmk}
\Cref{maintheorem} still holds if we replace $[0,1]$ with any interval $[a,b]$. However, to reduce notation in this paper, we assume without loss of generality that the domain of $f$ is $[0,1]$.
\end{rmk}

We prove Theorem \ref{maintheorem} via the following set-up. Define the family of maps
$\{\Phi_\alpha\}_{\alpha \in \R}$ associated to the curve $\Gamma$ by
\begin{align*}
\Phi_\alpha:[\alpha-1,\alpha]\times\R &\to \{\alpha\}\times \R
\\
(x_1, x_2) &\mapsto (\alpha,x_2+f(\alpha-x_1)).
\end{align*}
In other words, for $x$ in the plane, $\Phi_\alpha(x)$ is the unique point of intersection between the shifted curve $\Gamma +x$ and the vertical line $\{\alpha\}\times\R$.  We will also abuse notation and consider $\Phi_\alpha$ as a map into $\R$, by identifying $\R$ with the vertical line $\{\alpha\}\times \R$. 
This $1$-parameter family of maps was originally introduced by the third listed author and K.~Simon in \cite{STdim} to study the dimension and measure of sets of the form $A+\Gamma$, and to study the interior of such sets in \cite{STint}. 
Subsequently, the theory of these maps was further developed and applied to obtain upper and lower asymptotic bounds on a nonlinear variant of the Favard length problem (also known as the Buffon needle problem) for the four-corner Cantor set \cite{CDT22} (by L.~Cladek, B.~Davey and the third listed author) and for more general purely unrectifiable $1$-sets \cite{DT22} (by B.~Davey and the third listed author). 
See also the work of R.~Bongers and the third listed author \cite{BT23}, where related theory is developed for transversal families of maps, and it is shown that Marstrand's projection theorem holds when orthogonal projection maps are replaced by the family $\{\Phi_\alpha\}$. 

By Fubini's theorem, \Cref{maintheorem} is equivalent to the following: 
\begin{thm}[Non-linear prescribed projection theorem]
\label{prescribedprojection} 
For each $\alpha \in \R$, suppose $\newA_\alpha\subset \R$ is such that $\bigcup_\alpha (\{\alpha\}\times \newA_\alpha)$ is measurable.  
There exists a set $E\subset \R^2$ such that for almost every $\alpha$,
\begin{enumerate}[label=(\alph*)]
\item $\Phi_\alpha(E)\supset \newA_\alpha$,
\item $|\Phi_\alpha(E)\setminus \newA_\alpha|=0$.
\end{enumerate}
\end{thm}

It is interesting to compare and contrast Theorems \ref{maintheorem} and \ref{prescribedprojection} with the aforementioned results of Davies and Falconer.  The techniques used in this paper draw considerable inspiration from Falconer's proof of the prescribed projection theorem for orthogonal projections \cite{Fal86_paper}.  However, the class of generalized projections considered in this paper does not actually include orthogonal projections.  Indeed, it is not hard to see that Theorem \ref{maintheorem} could not possibly hold if $\Gamma$ is a line; by Fubini, any family of translates of a line that has positive measure on one slice must have positive measure on all (parallel) slices.  On the other hand, the family of all lines does not coincide with the family of translates of any single line. (It is the family of translations and rotations of any single line, but our setup does not allow rotations.)  Thus, our result is not a \emph{generalization} of Davies's theorem; it is a related but new result, and new ideas are required.  Both of these results are partial answers to the following general question.
\begin{question}
What families of plane curves $\cC$ have the property that for any measurable $\newA\subset \R^2$, there is a sub-family $\cC'\subset \cC$ such that $\newA\subset \cup \cC'$ and $|(\cup \cC')\setminus \newA|=0$? 
\end{question} 
Davies's Theorem shows that the family of lines has this property, while Theorem \ref{maintheorem} shows that the family of translates of a fixed curve has this property, under certain geometric assumptions on the curve. In both cases, the family of curves is two-dimensional, which might be the critical parameter dimension for this type of problem. It would be very interesting to develop a complete description of the families of curves with this property which unifies these partial results.

We end this subsection with several remarks to give greater context for our results.
\begin{rmk}
    As a concrete example, note that \Cref{maintheorem} applies to the case where $\Gamma$ is the upper quarter of the unit circle.  It is an open question to determine whether Theorem \ref{maintheorem} can be extended to the case where $\Gamma$ is the entire unit circle.  In addition to the circle not being the graph of a single function, this would require handling the case where the tangent slopes are unbounded.  
\end{rmk}

\begin{rmk}
As another concrete example, if we take $\newA$ to be the $y$-axis and $\Gamma$ to be the graph of any $C^1$ function $[0,1]\to\R$ with a strictly monotone derivative, then by using the fact that $\Gamma$ is a graph over the $x$-axis, \Cref{maintheorem} produces a Besicovitch-type set for $\Gamma$, i.e., a one-dimensional family of translates of $\Gamma$ whose union
has zero area. For such $\Gamma$, this is an improvement over \cite[Theorem 6.2]{CC} by the first author and Cs\"ornyei, which yields a one-dimensional family of translates of full-length subsets of $\Gamma$ whose union has zero area.
\end{rmk}
\begin{rmk}
We can say the following about the relationship between the sets $\newA$ and $E$ in Theorem \ref{maintheorem}. 
It follows from the main results in \cite{STdim} that 
$\dim_{\rm H}
(E+\Gamma) = 
\min\{\dim_{\rm H}(E) + \dim_{\rm H}(\Gamma), 2\}$; and
if $E$ is a $1$-set, then $|E+\Gamma|=0$ if and only if $E$ is purely unrectifiable.
It follows that if $\newA$ is a set of positive measure and the diameter of $\newA$ is $< 1$, 
then the set $E$ constructed by the proof of \Cref{maintheorem} is a purely unrectifiable set of dimension $1$ that is not of $\sigma$-finite length.
\end{rmk}

\begin{rmk}
\label{general_classes}

There are some families $\cC$ for which efficient covering theorems can be proved as simple corollaries of Davies' Theorem:
\begin{enumerate}
    \item 
If $\cC$ is the class of circles (not of fixed radius), then we may use the complex inversion map $z\mapsto 1/z$ to turn efficient coverings by lines into efficient coverings by circles of \textit{varying radii}.  One simply needs to ensure that lines in Davies' theorem stay away from the origin and use the fact that this map turns lines not through the origin into circles of varying radii.  In this case, $\cC$ is a $3$-dimensional family, whereas the family of all lines is $2$-dimensional; thus, we might expect it to be ``easier'' to cover with circles than with lines. By contrast, our main theorem allows efficient coverings by arcs of the \emph{unit} circle (but not by shifts of the entire unit circle). Because the radius is fixed, this family has dimension $2$, like the family of lines.   
\item 
Another interesting example is the family of parabolas $\cC=\{P(a,b):a,b\in\R\}$, where
\[
P(a,b)=\{x\in\R^2:x_2=x_1^2+ax_1+b\}.
\]
The existence of efficient coverings by parabolas can be proved as a corollary of Davies' theorem; the proof is found in Appendix \ref{parabola}.
\end{enumerate} 
\end{rmk}

\subsection{Notation and elementary calculations}
\label{notation}
 Throughout, fix a function $f\in C^1([0,1])$ such that $f'$ is strictly monotone, and let 
 \[
 \Gamma:=\{(t,f(t)):t\in[0,1]\}
 \]
 be its graph. The projections $(\Phi_\alpha)_{\alpha \in \R}$ associated to the curve $\Gamma$ are given by
\begin{align*}
\Phi_\alpha:[\alpha-1,\alpha]\times\R &\to \{\alpha\}\times \R
\\
(x_1, x_2) &\mapsto (\alpha,x_2+f(\alpha-x_1)).
\end{align*}
In other words, $\Phi_\alpha(x)$ is the unique point of intersection between the curve $\Gamma +x$ and the vertical line $\{\alpha\}\times\R$.  The preimage of a point $y\in \{\alpha\}\times\R$ is therefore $-\Gamma+y$.  Finally, we define 
\[
\Phi(E)=\bigcup_\alpha \Phi_\alpha(E\cap\dom\Phi_\alpha)=\Gamma+E.
\]
If $x\in\R^2$, we write $\Phi(x)=\Phi(\{x\})$.  \\

We will use the notation
\[
\dom \Phi_\alpha=[\alpha-1,\alpha]\times\R
\]
to refer to the domain of the projection $\Phi_\alpha$.  Many of our lemmas will contain the assumption that a given compact set is contained in the interior of the domain of $\Phi_\alpha$, so we also use the notation
\[
\intdom \Phi_\alpha=(\alpha-1,\alpha)\times\R.
\]

Given $\delta>0$ and a set $E\subset \R^2$, we write $\nbhd{E}{\delta}$ for the open $\delta$-neighborhood of $E$, and we write $B(x,r)$ for the ball centered at $x\in\R^2$ of radius $r>0$. \\

Throughout, we shall use the notation $X\lesssim Y$ to mean $X\leq CY$ for some positive constant $C$, with $C$ possibly depending on $f$ but not on any other parameters.  The notation $X\gtrsim Y$ is defined analogously, and $X\approx Y$ means $X\lesssim Y$ and $X\gtrsim Y$. \\

We consider the real projective line, which we identify with $\R/\pi\Z$ and think of as the space of \emph{unoriented} directions in $\R^2$. Given $\theta_1, \theta_2 \in \R/\pi\Z$, we let $[\theta_1, \theta_2] \subset \R/\pi\Z$ denote the closed arc from $\theta_1$ to $\theta_2$, using the standard (counterclockwise) orientation of the real projective line. Equivalently, 
\begin{align}
\label{eq:def interval in P1}
\theta \in [\theta_1, \theta_2] \iff \text{$\theta_1, \theta_2, \theta$ have representatives $\widetilde \theta_1, \widetilde \theta_2, \widetilde \theta \in \R$ such that $\widetilde \theta_1 \leq \widetilde \theta \leq \widetilde \theta_2$ and $\widetilde \theta_2 - \widetilde \theta_1 < \pi$.}
\end{align}
We define open arcs $(\theta_1, \theta_2)$ similarly. We define $|\theta_1 - \theta_2| \in [0,\frac{\pi}{2}]$ via the natural metric on $\R/\pi\Z$ induced by the metric on $\R$. \\

Suppose $x\in [\alpha-1,\alpha]\times\R$ and $y=\Phi_\alpha(x)$.  Since $f$ is differentiable, there is a well-defined tangent to the curve $-\Gamma +y$ at the point $x$; we denote the direction of this tangent by $\theta_\alpha(x) \in \R/\pi\Z$.  Explicitly, $\theta_\alpha(x)$ is the (unoriented) direction of the vector $(1,f'(\alpha-x_1))$.

It is often useful to identify the line $\{\alpha\}\times \R$ with simply $\R$, and to view $\Phi_\alpha$ as a projection from a strip in the plane to the real line.  By abuse of notation, we will write things like $\Phi_\alpha(x)\leq \Phi_\alpha(y)$ to mean the corresponding inequality under this identification; in other words, the point $\Phi_\alpha(y)$ lies above $\Phi_\alpha(x)$ vertically.  We will also denote the gradient of this map by $\nabla \Phi_\alpha$:
\[
\nabla \Phi_\alpha(x)=(-f'(\alpha-x_1),1).
\]
By continuity of $f'$, we have $|\nabla \Phi_\alpha(x)|\approx 1$.  The vector $\nabla \Phi_\alpha(x)$ is orthogonal to the direction $\theta_\alpha(x)$. Since $f'$ is strictly monotone, it follows that for each fixed $x$ the angle $\theta_\alpha(x)$ is monotone in $\alpha$.

\subsection{Outline of the proof strategy}
\label{outlinesection}
The proof of Theorem \ref{maintheorem} is based on what is called a ``Venetian blind'' construction.  The purpose of this section is to give the reader a high-level overview of this construction, before getting into the technical details in Section \ref{VenetianBlindSection}.  We start by illustrating Venetian blinds in the linear setting.  For $\theta\in \R/\pi\Z$, let $\pi_\theta$ denote the linear projection with fibers having angle $\theta$; thus, under this convention, $\pi_0$ is projection onto the vertical axis (the fibers are horizontal), and $\pi_{\pi/2}$ is projection onto the horizontal axis (the fibers are vertical).  If $L$ is a horizontal line segment, then $\pi_0(L)$ is a single point, $\pi_{\pi/2}(L)$ is a segment of length $|L|$, and for intermediate $\theta$, $\pi_\theta(L)$ is a segment of intermediate length (Figure \ref{A}). 
\begin{figure}[h]
\centering
\begin{minipage}[b]{0.45\linewidth}
\begin{tikzpicture}
\draw (0,3)--(0,6);
\draw (0,2)--(2,0);
\draw (3,0)--(6,0);

\draw[ultra thick] (4,5)--(5,5);

\draw[blue, ->] (4.5,4.5)--(4.5,4);
\draw[red, ->] (3.5,5)--(3,5);
\draw[purple, ->] (3.8,4.5)--(3.4,4.1);

\draw[ultra thick, blue] (4,0)--(5,0);
\draw[fill, red, ultra thick](0,5)circle[radius=0.01];
\draw[ultra thick, purple] (.5,1.5)--(1,1);

\node[above right] at (5,5) {$L$};
\node[left, red] at (0,5) {small};
\node[below left, purple] at (.75, 1.25) {medium};
\node[below, blue] at (4.5,0) {large};
\end{tikzpicture}
\caption{Projections of a horizontal segment}
\label{A}
\end{minipage}
\qquad
\begin{minipage}[b]{0.45\linewidth}
\begin{tikzpicture}
\draw (0,3)--(0,6);
\draw (0,2)--(2,0);
\draw (3,0)--(6,0);

\draw[ultra thick] (4,5)--(4.1,5.1);
\draw[ultra thick] (4.1,5)--(4.2,5.1);
\draw[ultra thick] (4.2,5)--(4.3,5.1);
\draw[ultra thick] (4.3,5)--(4.4,5.1);
\draw[ultra thick] (4.4,5)--(4.5,5.1);
\draw[ultra thick] (4.5,5)--(4.6,5.1);
\draw[ultra thick] (4.6,5)--(4.7,5.1);
\draw[ultra thick] (4.7,5)--(4.8,5.1);
\draw[ultra thick] (4.8,5)--(4.9,5.1);
\draw[ultra thick] (4.9,5)--(5,5.1);

\draw[blue, ->] (4.5,4.5)--(4.5,4);
\draw[red, ->] (3.5,5)--(3,5);
\draw[red, ->] (3.8,4.5)--(3.4,4.1);

\draw[ultra thick, blue] (4,0)--(5,0);
\draw[ultra thick, red] (0,5)--(0,5.2);
\draw[fill, red, ultra thick](.5,1.5)circle[radius=0.01];

\draw[fill, red, ultra thick](.6,1.4)circle[radius=0.01];

\draw[fill, red, ultra thick](.7,1.3)circle[radius=0.01];

\draw[fill, red, ultra thick](.8,1.2)circle[radius=0.01];

\draw[fill, red, ultra thick](.9,1.1)circle[radius=0.01];

\node[above right] at (5,5) {$L$};
\node[left, red] at (0,5) {small};
\node[below left, red] at (.75, 1.25) {small};
\node[below, blue] at (4.5,0) {cover};
\end{tikzpicture}
\caption{Extending the ``small projection'' range}
\label{B}
\end{minipage}
\end{figure}
 Now, suppose we replace $L$ with $N$ segments $L_1,\dots,L_N$ as in Figure \ref{C}, so that:
\begin{itemize}
\item The segments $L_i$ each have one ``initial'' endpoint on $L$ and one ``terminal'' endpoint above $L$,
\item The initial endpoints of the new segments are evenly spaced along $L$,
\item The segments $L_i$ make angle $\theta_0$ with $L$,
\item The segments $L_i$ are just long enough that the terminal endpoint of $L_i$ lies directly above the initial endpoint of $L_{i+1}$.
\end{itemize}

\begin{figure}[h]
\centering
\begin{tikzpicture}
\draw[thick] (0,0)--(8,0);

\draw[thick] (0,0)--(1,1);
\draw[thick] (1,0)--(2,1);
\draw[thick] (2,0)--(3,1);
\draw[thick] (3,0)--(4,1);
\draw[thick] (4,0)--(5,1);
\draw[thick] (5,0)--(6,1);
\draw[thick] (6,0)--(7,1);
\draw[thick] (7,0)--(8,1);

\draw (.4,0)to[out=90,in=135](.25,.25);

\node[right] at (8,0) {$L$};
\node[above] at (1,1) {$L_1$};
\node[above] at (2,1) {$L_2$};
\node[above] at (8,1) {$L_N$};
\node[right] at (.4,.2) {$\theta_0$};
\end{tikzpicture}
\caption{Venetian blinds}
\label{C}

\end{figure}
When we project onto the $x_1$-axis, the projection is unchanged.  Thus, we refer to the direction $\theta=\pi/2$ as a ``covering'' direction, since we have the covering property
\[
\pi_{\pi/2}(L_1\cup\cdots\cup L_N)\supset \pi_{\pi/2}(L).
\]
If we project onto the $x_2$ axis instead, the projection has gone from a single point to an interval; however, this interval can be made as small as desired by making $N$ large.  Thus, the direction $\theta=0$ is still a ``small'' direction (as are directions $\theta$ sufficiently close to $0$).  However, if we project in direction $\theta=\theta_0$, the image is a finite set of points.  Therefore, we have extended the range of ``small'' directions to include $\theta=\theta_0$ (Figure \ref{B}), as well as directions $\theta$ sufficiently close to $\theta_0$.  Now, suppose we make $\theta_0$ very small, and iterate the construction.  At each iteration, we increment the angle by the same small amount, until we get to (for example) an angle $\pi/4$.  If the angle increment is sufficiently small, we can ensure that all directions $\theta\in [0,\pi/4]$ are ``small projection'' directions, and all directions $\theta\in [\pi/4+\delta,\pi/2]$ are ``covering'' directions (for some small $\delta$).  Versions of this construction are the basis of prescribed projection results.  The goal is to prove lemmas of the following forms:
\begin{itemize}
\item \textbf{Small projections lemma:} Let $\e>0$, let $L$ be a line segment, and let $\Theta_\sma\subset \R/\pi\Z$ be compact.  There exist parameters for the iterated Venetian blind construction which produce a finite collection of line segments $\mathcal{V}$ satisfying $|\pi_\theta(\cup \mathcal{V})|<\e|L|$ whenever $\theta\in \Theta_\sma$.
\item \textbf{Covering lemma:} Let $L$ be a line segment, and let $\Theta_\cov\subset \R/\pi\Z$ be compact.  There exist parameters for the iterated Venetian blind construction which produce a finite collection of line segments $\mathcal{V}$ satisfying $\pi_\theta(\cup\mathcal{V})\supset \pi_\theta(L)$ whenever $\theta\in \Theta_\cov$.
\item \textbf{Combined lemma:} Under some conditions on the sets $\Theta_\sma$ and $\Theta_\cov$, the parameters can be chosen to simultaneously give a small projection lemma and a covering lemma.
\end{itemize}
In this paper, our goal is to extend this scheme to curve projections.  Let $\Phi_\alpha(x)$ and $\theta_\alpha(x)$ be as defined in Section \ref{notation}.  For a given $x\in \R^2$ and $y=\Phi_\alpha(x)$, we will be interested in the fiber of $x$ under the projection $\Phi_\alpha$, namely $\Phi_\alpha^{-1}(y)=y-\Gamma$.  Recall that $\theta_\alpha(x)$ is the tangent to this curve at the point $x$.  Therefore, locally, the projection $\Phi_\alpha(x)$ acts like a linear projection in the direction $\theta_\alpha(x)$.  If $\theta_0$ is a ``covering'' direction for our Venetian blind construction and $\theta_\alpha(x)\sim \theta_0$ for all $x$ on the line segment, then we will have a covering statement for the projections $\Phi_\alpha$.  Geometrically, this means that any fiber which intersects the line segment must also intersect the blinds if the tangents are all in the covering directions (Figure \ref{D}).  On the other hand, if $\alpha$ is such that $\theta_\alpha(x)$ is a small direction for all $x\in L$, the measure $|\Phi_\alpha(L)|$ of the projection will be small.

\begin{figure}[h]
\centering
\begin{tikzpicture}
\draw[thick] (0,0)--(8,0);

\draw[thick] (0,0)--(1,1);
\draw[thick] (1,0)--(2,1);
\draw[thick] (2,0)--(3,1);
\draw[thick] (3,0)--(4,1);
\draw[thick] (4,0)--(5,1);
\draw[thick] (5,0)--(6,1);
\draw[thick] (6,0)--(7,1);
\draw[thick] (7,0)--(8,1);

 \draw [blue,thick,domain=208:250] plot ({9.1+4*cos(\x)}, {3.2+4*sin(\x)});
 \draw [red,thick,domain=290:330] plot ({4*cos(\x)}, {3.2+4*sin(\x)});

\node[right] at (8,0) {$L$};
\node[above] at (1,1) {$L_1$};
\node[above] at (2,1) {$L_2$};
\node[above] at (8,1) {$L_N$};
\node[right, blue] at (10.4,-.1) {cover};
\node[right, red] at (10.4,1.1) {small};

\path[fill=blue!20](10.72, .5)to[out=270,in=0](10,-.22)--(10,.5)--(10.72,.5);
\path[fill=red!20](10.72, .5)to[out=90,in=300](10.4,1.1)--(10,.5)--(10.72, .5);
\draw[thick, blue, ->] (10,.5)--(10.4,-.1);
\draw[thick, red, ->] (10,.5)--(10.4,1.1);

\end{tikzpicture}
\caption{Fibers of curve projections with tangents in covering directions and small directions.   }
\label{D}
\end{figure}

If $x$ is fixed and $\Theta_\sma$ and $\Theta_\cov$ are sets of directions, we have corresponding sets of our parameter $\alpha$:
\begin{align*}
A_\sma&=\{\alpha\in \R:\theta_\alpha(x)\in \Theta_\sma\}, \\
A_\cov&=\{\alpha\in \R:\theta_\alpha(x)\in \Theta_\cov\}.
\end{align*}
Since $\theta_\alpha(x)$ is continuous and monotone, assumptions on $\Theta_\cov$ and $\Theta_\sma$ turn into assumptions on $A_\cov$ and $A_\sma$, and vice versa.  For example, $A_\cov$ is an interval if and only if $\Theta_\cov$ is an interval, $A_\sma$ is compact if and only if $\Theta_\cov$ is compact, and so on.  The main work of this paper, undertaken in Sections \ref{VenetianBlindSection} and \ref{KeyLemmaSection}, is to prove non-linear versions of the lemmas discussed above:
\begin{itemize}
\item \textbf{Non-linear small projections lemma:}  Let $\e>0$, let $L$ be a line segment, and let $A_\sma  \subset \R$ be compact.  There exist parameters for the iterated Venetian blind construction which produce a finite collection of line segments $\mathcal{V}$ satisfying $|\Phi_\alpha(\cup\mathcal{V})|<\e|L|$ whenever $\alpha\in A_\sma$.
\item \textbf{Non-linear covering lemma:}  Let $L$ be a line segment, and let $A_\cov\subset \R$ be compact.  There exist parameters for the Venetian blind construction which produce a finite collection of line segments $\mathcal{V}$ satisfying $\Phi_\alpha(\cup\mathcal{V})\supset \Phi_\alpha(L)$ whenever $\alpha\in A_\cov$.
\item \textbf{Non-linear combined lemma:}  Under some conditions on the sets $A_\sma$ and $A_\cov$, the parameters can be chosen to simultaneously give a non-linear small projections lemma and a non-linear covering lemma.
\end{itemize}

While our construction uses the same Venetian blind scheme as in the linear setting, proving non-linear versions of these lemmas requires many new ideas which distinguish this paper from previous work.  One difficulty in our present setting which must be accounted for is the fact that the domain of our curve projections $\Phi_\alpha$ is not all of $\R^2$, but rather the vertical strip $[\alpha-1,\alpha]\times \R$.  This complicates our ability to prove a non-linear covering lemma, because covering can fail due to domain issues rather than direction issues.  If $L$ and $\alpha$ are such that $L$ intersects the boundary of $\dom \Phi_\alpha$, then even if the directions $\theta_\alpha(x)$ are in covering directions, a fiber which intersects $L$ may not intersect any of the blinds simply because it ends before reaching them (Figure \ref{E}).  Moreover, because the proof is based on iteration, we have to consider many different segments and many different values of the parameter $\alpha$ at once, so the issue is not resolved by simply assuming that one particular segment is contained in one particular strip.  Instead, many of the restrictions in our lemmas proved in Sections \ref{VenetianBlindSection} and \ref{KeyLemmaSection} are forced on us by the need to handle the domain issue.
\begin{figure}[h]
\centering
\begin{tikzpicture}
\draw[thick] (0,0)--(8,0);

\draw[thick] (0,0)--(1,1);
\draw[thick] (1,0)--(2,1);
\draw[thick] (2,0)--(3,1);
\draw[thick] (3,0)--(4,1);
\draw[thick] (4,0)--(5,1);
\draw[thick] (5,0)--(6,1);
\draw[thick] (6,0)--(7,1);
\draw[thick] (7,0)--(8,1);

 \draw [blue,thick,domain=210:230] plot ({8.8+6*cos(\x)}, {3.2+6*sin(\x)});

\node[right] at (8,0) {$L$};
\node[above] at (1,1) {$L_1$};
\node[above] at (2,1) {$L_2$};
\node[above] at (8,1) {$L_N$};
\node[right, blue] at (10.4,-.1) {cover};
\node[right, red] at (10.4,1.1) {small};

\path[fill=blue!20](10.72, .5)to[out=270,in=0](10,-.22)--(10,.5)--(10.72,.5);
\path[fill=red!20](10.72, .5)to[out=90,in=300](10.4,1.1)--(10,.5)--(10.72, .5);
\draw[thick, blue, ->] (10,.5)--(10.4,-.1);
\draw[thick, red, ->] (10,.5)--(10.4,1.1);
\end{tikzpicture}
\caption{Venetian blinds and fibers of curve projections}
\label{E}
\end{figure}

The remainder of the paper is organized as follows.  In Section \ref{VenetianBlindSection}, we rigorously define the Venetian blind set and iterated Venetian blind set corresponding to an initial segment and a choice of parameters, and prove small projections lemmas and covering lemmas for each stage of the construction.  In Section \ref{KeyLemmaSection}, we prove a key lemma which can be understood as a stepping stone to our main theorem, and forms the basis for its proof.  Finally, in  \Cref{ProofSection}, we prove a series of lemmas which establish a mechanism for converting the key lemma (\Cref{lemma: keylemma}) in Section \ref{KeyLemmaSection} into an efficient covering theorem, and then prove Theorem \ref{maintheorem}.

\subsection{Acknowledgments}
We thank Kenneth Falconer for introducing us to this area, and we thank Falconer and K\'aroly Simon for productive discussions related to this article. We also thank Damian D\k{a}browski and Tuomas Orponen for several helpful comments in preparing the draft. We thank Robert Fraser and Mark Meyer for pointing out a gap in an earlier version of \Cref{base}. We thank the anonymous referee for independently pointing out the same gap, as well as for many helpful suggestions that improved the paper in many ways.

\section{The Venetian blind construction}
\label{VenetianBlindSection} 
Given two distinct points $a,b \in \R^2$, we define the (oriented) line segment
\[
\LINE(a,b) = \{(1-t)a + tb : t \in [0,1]\}.
\]
For a line segment $L = \LINE(a,b)$, we let $|L|$ denote the length of $L$ and $\theta_L \in \R/\pi\Z$ denote the direction of $L$.

In the context of linear projections, the segment $\LINE(a,b)$ maps to a single point when projected in the direction $\theta_{\LINE(a,b)}$, and has small length if projected in directions close to $\theta_{\LINE(a,b)}$.  In the general projection setting, we have the following version of this principle.
\begin{lem}[Small projections in certain directions]
\label{lemma:segment-small-projection-nonlinear}
For any line segment $L = \LINE(a,b)$ and all $\alpha \in \R$,
\[
|\Phi_\alpha(L \cap \dom \Phi_\alpha)| 
\lesssim
\left(\sup_{z \in L \cap \dom \Phi_\alpha} |\theta_{\alpha}(z) - \theta_L|\right) |L|.
\]
\end{lem}

\begin{proof}
By the convexity of $\dom \Phi_\alpha$ and the mean value theorem, for every $x,y \in L \cap \dom \Phi_\alpha$, there exists $z \in L \cap \dom \Phi_\alpha$ such that
\[
|\Phi_\alpha(x)-\Phi_\alpha(y)|=|\nabla \Phi_\alpha(z)\cdot (x-y)|
=
|\nabla \Phi_\alpha(z)| |x-y|
\sin|\theta_\alpha(z) - \theta_L|
\]
In the second equality above, we used the fact that the two vectors in the dot product have directions $\theta_\alpha(z)+\frac{\pi}{2}$ and $\theta_L$. Finally, we note that $|\nabla \Phi_\alpha(z)|\approx 1$ (by the assumption $f\in C^1$) and $|x-y|\leq |L|$, so
\[
|\Phi_\alpha(x)-\Phi_\alpha(y)|
\lesssim
 |L| \cdot |\theta_{\alpha}(z) - \theta_L|,
\]
which proves the lemma.
\end{proof}

\subsection{Dividing}

Given a line segment $\LINE(a,b)$ and $N \in \N$, we define 
\[
\DIVIDE(\LINE(a,b), N)
=
\{
\LINE(a_{i-1}, a_{i}) : i = 1, \ldots, N
\},
\qquad
\text{where }
a_i = (1-\tfrac{i}{N}) a + \tfrac{i}{N} b.
\]
In other words, $\DIVIDE(\cdots)$ divides the line segment into $N$ equally sized line segments.

\subsection{Rotating}
Given a line segment $L=\LINE(a,b)$ and two directions $\theta_{\cov}, \theta_{\sma} \in \R/\pi\Z$ such that $\theta_L, \theta_{\sma}, \theta_{\cov}$ are distinct, we define
\[
\ROTATE(\LINE(a,b), \theta_{\sma} , \theta_{\cov})
=
\LINE(a,c),
\]
where $c \in \R^2$ is the unique point satisfying the following:
(see \Cref{F}):
\begin{align}
\label{eq:triangle-small-cover}
\theta_{\LINE(a,c)} = \theta_{\sma}, \qquad
\theta_{\LINE(b,c)} = \theta_{\cov}
\end{align}
\begin{figure}[h]
\centering
\begin{tikzpicture}
\draw[thick, -] (0,0)--(3,0);
\draw[thick, -] (0,0)--(2.5,1);
\draw[thick, -] (2.5,1)--(3,0);

\draw[fill, ultra thick](0,0)circle[radius=0.05];
\draw[fill, ultra thick](3,0)circle[radius=0.05];
\draw[fill, ultra thick](2.5,1)circle[radius=0.05];

\node[below] at (1.5,0) {$\theta_L$};
\node[above left] at (1.25,.5) {$\theta_\sma$};
\node[above right] at (2.75,.5) {$\theta_\cov$};

\node[below left] at (0,0) {$a$};
\node[below right] at (3,0) {$b$};
\node[above] at (2.5,1.05) {$c$};

\end{tikzpicture}
\caption{Defining the point $c$.}
\label{F}
\end{figure}

Note that the ``rotated'' segment does not have the same length as the initial segment.  We have the following geometric relationship between the rotation and the original segment:
\begin{lem}[Basic geometry of $\ROTATE$]
\label{lemma: ROTgeo}
Let $L=\LINE(a,b)$, and let $\theta_{\cov}, \theta_{\sma} \in \R/\pi\Z$ be such that $\theta_L, \theta_{\sma}, \theta_{\cov}$ are distinct.
Then:
\begin{enumerate}[label=(\alph*)]
\item We have
\begin{align}
|\ROTATE(L, \theta_{\sma} , \theta_{\cov})|
=
\frac{\sin |\theta_L-\theta_\cov|}{\sin|\theta_\sma-\theta_\cov|}  |L|
\end{align}
\item Let $\delta=2\frac{\sin |\theta_\sma-\theta_L|}{\sin|\theta_\sma-\theta_\cov|} |L|$.  Then,
\[
\ROTATE(L,\theta_\sma,\theta_\cov)\subset \nbhd{L}{\delta}
\]
\end{enumerate}

\end{lem}
\begin{proof}
By \eqref{eq:triangle-small-cover} and the law of sines, we have
\[
\frac{|\ROTATE(L, \theta_{\sma} , \theta_{\cov})|}{|L|}
=
\frac{|c-a|}{|b-a|}
=
\frac{\sin |\theta_L-\theta_\cov|}{\sin|\theta_\sma-\theta_\cov|} .
\]
This proves part (a).  For (b), we have
\begin{align*}
\max
\{d(x,L) : x \in \ROTATE(L,\theta_\sma,\theta_\cov)\}
=
d(c,L)
\leq
|c-b|=\frac{\sin |\theta_\sma-\theta_L|}{\sin|\theta_\sma-\theta_\cov|} |L|
=\frac{1}{2}\delta,
\end{align*}
where we used the law of sines again. (Note that the factor of $2$ in the definition of $\delta$ is needed because we define $\nbhd{L}{\delta}$ to be an open neighborhood.)
\end{proof}

\begin{lem}[Covering property for one segment]
\label{lemma:rotation-cover-nonlinear}
Let $L=\LINE(a,b)$, and let $\theta_\sma, \theta_\cov \in \R/\pi\Z$ satisfy $\theta_L \in (\theta_{\cov}, \theta_{\sma})$. (See \eqref{eq:def interval in P1} for the precise meaning of this.) Let $c \in \R^2$ be such that \[\ROTATE(L,\theta_{\sma},\theta_{\cov})=\LINE(a,c)\] and let $\Delta$ be the convex hull of $\{a,b,c\}$.
Suppose $\alpha$ is such that $\Delta \subset \dom \Phi_\alpha$, and for all $x \in \Delta$ we have $\theta_\alpha(x)\in [\theta_{\cov}, \theta_L]$.  Then,
\[
\Phi_\alpha(L)
\subset
\Phi_\alpha (\ROTATE(L, \theta_{\sma} , \theta_{\cov})).
\]
\end{lem}

\begin{proof}
For every $\theta \in [\theta_{\cov}, \theta_L]$ the vectors $a-b$ and $c-b$ lie on opposite sides of the line with direction $\theta$. (This is by definition of $c$ and the fact that $\theta_L \in (\theta_\cov, \theta_\sma)$; see \Cref{figure:same sign}.) For every $x \in \Delta$, $\nabla \Phi_\alpha(x)$ is perpendicular to $\theta_\alpha(x) \in [\theta_\cov, \theta_L]$, so $\nabla \Phi_\alpha(x) \cdot (b-a)$ and $\nabla \Phi_\alpha(x) \cdot (c-b)$ have the same sign. Combining this observation with the continuity of $\Phi_\alpha$, we see that the following set of values
\[
\{ \nabla \Phi_\alpha(x) \cdot (b-a) : x\in \Delta\}
\cup
\{ \nabla \Phi_\alpha(x) \cdot (c-b) : x\in \Delta\}
\]
is either contained in $[0, \infty)$ or $(-\infty,0]$. 

\begin{figure}[h]
\centering
\begin{tikzpicture}
\draw[thick, -] (-1,0)--(5,0);
\draw[thick, -] (0,0)--(2.5,1);
\draw[thick, -] (2.5,1)--(3,0);
\draw[thick, -] (3.8,-1.6)--(2.2,1.6);

\draw[dashed] (4.5,-1.1)--(1.5,1.1);

\draw[fill, ultra thick](0,0)circle[radius=0.05];
\draw[fill, ultra thick](3,0)circle[radius=0.05];
\draw[fill, ultra thick](2.5,1)circle[radius=0.05];

\node[right] at (5,0) {$\theta_L$};
\node[below] at (3.8,-1.6) {$\theta_\cov$};
\node[below right] at (4.5,-1.1) {$\theta \in [\theta_{\cov}, \theta_L]$};

\node[below left] at (0,0) {$a$};
\node[below left] at (3,0) {$b$};
\node[above right] at (2.5,1.05) {$c$};

\end{tikzpicture}
\caption{For every $\theta \in [\theta_{\cov}, \theta_L]$ the vectors $a-b$ and $c-b$ lie on opposite sides of the line with direction $\theta$.}
\label{figure:same sign}
\end{figure}

Without loss of generality, assume the set is contained in $[0,\infty)$. Then the map $t \mapsto \Phi_\alpha((1-t)a+tb)$ is monotonic increasing on $[0,1]$. Thus, it suffices to show
\[
\Phi_\alpha(a)
\leq
\Phi_\alpha(b)
\leq
\Phi_\alpha(c),
\]
since the lemma would then follow from the intermediate value theorem. The first inequality ($\Phi_\alpha(a)\leq\Phi_\alpha(b)$) follows from monotonicity. For the second ($\Phi_\alpha(b)\leq\Phi_\alpha(c)$), we use the mean value theorem together with the assumption $\LINE(b,c) \subset \dom \Phi_\alpha$ to obtain $z \in \LINE(b,c)$ such that
\begin{align*}
\Phi_\alpha(b)
-
\Phi_\alpha(c)
=
\nabla \Phi_\alpha(z) \cdot (b-c)
\leq 0.
\end{align*}
The inequality above follows from the assumption that $\{ \nabla \Phi_\alpha(x) \cdot (c-b) : x\in \LINE(b,c)\} \subset [0,\infty)$.
\end{proof}
The assumption in Lemma \ref{lemma:rotation-cover-nonlinear} that $\Delta \subset \dom\Phi_\alpha$ is necessary to handle the domain issue discussed at the end of Section \ref{outlinesection}; this assumption avoids the problem shown in Figure \ref{E}.  
\subsection{Venetian blinds}

Given a line segment $L$, a parameter $N\in\N$, and directions $\theta_\cov$ and $\theta_\sma$ satisfying $\theta_L \in (\theta_{\cov}, \theta_{\sma})$, we define the corresponding set of Venetian blinds by dividing the segment into $N$ pieces and rotating:
\[
\VB(L, \theta_{\sma}, \theta_{\cov}, N)
=
\{
\ROTATE(L', \theta_{\sma}, \theta_{\cov}) : L' \in \DIVIDE(L, N)
\}
\]
Note that $\VB(\cdots)$ denotes a set of line segments. We denote the union of the segments by $\cup\VB(\cdots)$. If $S$ is a finite union of line segments, we let $|S|$ denote the total length of $S$.  Lemma \ref{lemma: ROTgeo} has an immediate corollary for Venetian blinds:
\begin{lem}[Basic geometry of Venetian blinds]
\label{lemma: VBgeo}
Let $L$ be a line segment, and let $\theta_\sma, \theta_\cov \in \R/\pi\Z$ satisfy $\theta_L \in (\theta_{\cov}, \theta_{\sma})$. Then:
\begin{enumerate}[label=(\alph*)]
\item The lengths of the blinds are related to the length of the original segment by
\[
|\cup\VB(L,\theta_\sma,\theta_\cov,N)|=\frac{\sin|\theta_L - \theta_{\cov}|}{\sin|\theta_{\sma} - \theta_{\cov}|}|L|
\]
\item Let $\delta > 0$. Then for $N$ sufficiently large (depending on $L, \theta_{\sma}, \theta_{\cov}, \delta$), we have
\[
\cup\VB(L,\theta_\sma,\theta_\cov,N)\subset \nbhd{L}{\delta}.
\]
\end{enumerate}

\end{lem}
\begin{proof}
To prove (a), we have $\cup\VB(L,\theta_\sma,\theta_\cov,N)=\bigcup_{L_i\in\DIVIDE(L,N)}\ROTATE(L_i,\theta_\sma,\theta_\cov)$.  Applying Lemma \ref{lemma: ROTgeo} to each $L_i$ individually (note that $|L_i|=|L|/N$) immediately establishes part (a). To prove part (b), we use \Cref{lemma: ROTgeo}(b) to see that 
\[
\cup\VB(L,\theta_\sma,\theta_\cov,N)  \subset \nbhd{L}{\delta'},
\qquad \text{ where } \delta' = \frac{2|L| \sin |\theta_\sma-\theta_L|}{N\sin|\theta_\sma-\theta_\cov|}.
\qedhere
\]
\end{proof}

\begin{lem}[Covering property for Venetian blinds]
\label{lemma:vb-cover-nonlinear}
Let $L$ be a line segment, and let $\theta_\sma, \theta_\cov \in \R/\pi\Z$ satisfy $\theta_L \in (\theta_{\cov}, \theta_{\sma})$. 
Let $A\subset \R$ be a compact set such that for every $\alpha\in A$ we have $L\subset\intdom \Phi_\alpha$.  Suppose also that for each $\alpha\in A$ and $x\in L$, we have $\theta_\alpha(x)\in (\theta_\cov,\theta_L)$.  Then, if $N$ is sufficiently large (depending on $L, \theta_\sma,\theta_\cov, A$ but uniform over $\alpha\in A$), we have
\begin{align*}
&\cup\VB(L, \theta_{\sma}, \theta_{\cov}, N)
\subset
\intdom\Phi_\alpha
\qquad\text{ for all $\alpha \in A$}
\\
&\Phi_\alpha(L)
\subset
\Phi_\alpha (\cup\VB(L, \theta_{\sma}, \theta_{\cov}, N))
\qquad\text{ for all $\alpha \in A$}.
\end{align*}
\end{lem}
\begin{proof}

By compactness and continuity, there exists a $\delta > 0$ (depending on $L, \theta_\cov, A$) such that 
\begin{gather*}
\text{$\nbhd{L}{\delta} \subset \intdom\Phi_\alpha$ for all $\alpha \in A$}
\\
\text{$\theta_\alpha(x) \in (\theta_{\cov}, \theta_L)$ for all $x \in \nbhd{L}{\delta}$ and all $\alpha \in A$}.
\end{gather*}
If $N$ is sufficiently large (depending on $L, \theta_\sma,\theta_\cov, \delta$), then for each $L' \in \DIVIDE(L, N)$, the convex hull of $L' \cup \ROTATE(L',\theta_\sma,\theta_\cov)$ is contained in $\nbhd{L}{\delta}$, so the hypotheses of Lemma \ref{lemma:rotation-cover-nonlinear} are met.  Applying Lemma \ref{lemma:rotation-cover-nonlinear} to each segment in $\DIVIDE(L, N)$ then gives the desired result.
\end{proof}

\subsection{Iterated Venetian blinds}
We will index segments in our iterated Venetian blind construction by $\N^*=\bigcup_{k=0}^\infty \N^k$, the set of finite sequences of non-negative integers.  We first establish some notation and terminology for this index set.  If $\ii = (i_1, \ldots, i_k) \in \N^k$, we say $\ii$ has length $k$ and write $|\ii| = k$.  The empty sequence is the unique sequence of length $0$, and is denoted $\emp$.  If $k=|\ii|<|\jj|$ and $i_l=j_l$ for all $l\leq k$, we say that $\jj$ is a descendant of $\ii$ and $\ii$ is an ancestor of $\jj$.  Furthermore, if $|\jj| = |\ii| + 1$, we say $\ii$ is the parent of $\jj$ and that $\jj$ is a child of $\ii$.  If $\ii=(i_1,\dots,i_k)$ and $\jj=(i_1,\dots,i_k,j)$ is a child of $\ii$, we will use the notation $\jj=(\ii,j)$.  A \textbf{finite tree} of depth $m$ is a finite set $\mathcal{T}\subset \N^*$ such that the following hold:
\begin{itemize}
\item $\varnothing\in\cT$
\item If $\ii\in \mathcal{T}$, all ancestors of $\ii$ are also in $\mathcal{T}$
\item A node $\ii\in \cT$ is a leaf (i.e., a node with no children) if and only if $|\ii| = m$.
\end{itemize}
If we have a fixed finite tree $\cT$, for each $\ii\in\mathcal{T}$, we let $N_\ii$ denote the number of children of $\ii$ which lie in $\mathcal{T}$. \\  

Let $L_\emp$ be a line segment with direction $\theta_0 \in \R/\pi\Z$. Let directions $\theta_\cov, \theta_\sma \in \R/\pi\Z$ satisfy $\theta_0 \in (\theta_{\cov},\theta_{\sma})$.
Let $m\in\N$, and let $\cT$ be a finite tree of depth $m$. For $k=1,\dots,m$, define:
\[
\theta_k = \theta_0 + \frac{k}{m}(\theta_{\sma}-\theta_0).
\]
(Here, $\theta_{\sma}-\theta_0$ is interpreted as a value in $[0,\pi)$.) We construct a finite collection of line segments recursively as follows.  Our recursion starts with the initial segment $L_\emp$ above.  Given $L_\ii$ for some $\ii\in\mathcal{T},|\ii|<m$, we define $L_{(\ii,j)}$ by 
\[
\{L_{(\ii,j)} : j = 1, \ldots, N_\ii\} 
=
\VB(L_\ii, \theta_{|\ii|+1}, \theta_{\cov}, N_\ii).
\]
We define iterated Venetian blind to be the set of segments:
\[
\IVB(L_\emp, \theta_{\sma}, \theta_{\cov},
\cT
)
=
\{ L_{\ii} : |\ii| = m \}.
\]

\begin{rmk}
\label{rmk:sufficiently large}
In the next three lemmas, we will need to make the segments of the iterated Venetian blind ``sufficiently fine'' in order to obtain certain properties. In the statements of the lemmas, when we say 
\[
\text{``if the $N_\ii$ are recursively sufficiently large depending on $X$, then $Y$ holds''},
\]
we mean the following: For each $|\ii| < m$, there exists $N_\ii^* \in \N$, which can depend not only on $X$ but also (crucially) on $(N_\jj)_{|\jj| < |\ii|}$, such that if $N_\ii \geq N_\ii^*$ for all $\ii \in \cT$, then $Y$ holds.
\end{rmk}

\begin{lem}[Basic geometry of iterated Venetian blinds]
\label{lemma: IVBgeo}
Let $L_\emp$ be a line segment with direction $\theta_0 \in \R/\pi\Z$. Let directions $\theta_\cov, \theta_\sma \in \R/\pi\Z$ satisfy $\theta_0 \in (\theta_{\cov},\theta_{\sma})$.
Let $\cT$ be a finite tree of depth $m$.  Then:
\begin{enumerate}[label=(\alph*)]
\item For all $k=1,\dots,m$, we have
\[
\sum_{|\ii|=k}|L_\ii|\leq \max\left(1,\frac{\sin|\theta_0-\theta_\cov|}{\sin|\theta_\sma-\theta_\cov|}\right)|L_\emp|.
\]
\item Let $\delta > 0$. If the $N_\ii$ are recursively sufficiently large (depending on $L, \theta_{\cov}, \theta_{\sma}, m, \delta$) (see \Cref{rmk:sufficiently large}), then we have that for every $\ii$,
\[
L_\ii\subset\nbhd{L_\emp}{\delta}.
\] 
\end{enumerate}
\end{lem}
\begin{proof}
For (a), we apply part (a) of Lemma \ref{lemma: VBgeo} to get, for any $|\ii|=k$,
\[
\sum_{\substack{\jj \\ \jj \text{ a child of } \ii}} |L_\jj|=\frac{\sin|\theta_k-\theta_\cov|}{\sin|\theta_{k+1}-\theta_\cov|}|L_\ii|.
\]
Therefore,
\[
\sum_{|\jj|=k+1}|L_\jj|=\frac{\sin|\theta_k-\theta_\cov|}{\sin|\theta_{k+1}-\theta_\cov|}\sum_{|\ii|=k}|L_\ii|.
\]
By telescoping, for all $k=0,\dots,m$,
\begin{align}
\label{eq:sum L_i bound}
\sum_{|\ii|=k}|L_\ii|=\frac{\sin|\theta_0-\theta_\cov|}{\sin|\theta_k-\theta_\cov|}|L_\emp|.
\end{align}
By elementary considerations, it is easy to verify
\[
\frac{\sin|\theta_0-\theta_\cov|}{\sin|\theta_k-\theta_\cov|}\leq\max\left(1,\frac{\sin|\theta_0-\theta_\cov|}{\sin|\theta_\sma-\theta_\cov|}\right),
\]
so the proof of part (a) is complete.  

To prove part (b), let $\delta > 0$. We will show that if the $N_\ii$ are recursively sufficiently large depending on $L, \theta_{\cov}, \theta_{\sma}, m, \delta$ (see \Cref{rmk:sufficiently large} for the precise meaning), then we have that for all $\ii \in \cT$,
\begin{align}
\label{eq:VB in nbhd}
\cup\VB(L_\ii, \theta_{|\ii|+1}, \theta_{\cov}, N_\ii)
\subset\nbhd{L_\emp}{\delta}.
\end{align} 
Let $k \in \{0,\ldots, m-1\}$. Suppose $(N_\ii)_{|\ii| < k}$ are sufficiently large so that \eqref{eq:VB in nbhd} holds for all $|\ii| < k$. Now fix some $\ii$ with $|\ii| = k$. We have $L_\ii \subset \intdom \Phi_\alpha$ for all $\alpha \in A$. (For $k=0$, this is just the statement $L_\emp \subset\nbhd{L_\emp}{\delta}$. For $k > 0$, this follows by applying  \eqref{eq:VB in nbhd} to the parent of $\ii$.) Thus, we can apply \Cref{lemma: VBgeo}(b) to $L_\ii$ to see that if $N_\ii$ is sufficiently large (depending on $L_\ii, \theta_{\sma}, \theta_{\cov},m, \delta$), then \eqref{eq:VB in nbhd} holds for our fixed $\ii$.

Since every segment $L_\ii$ except for $L_\emp$ arises in one of the Venetian blind constructions, this proves part (b).
\end{proof}

\begin{lem}[Covering property for iterated Venetian blinds]
\label{lemma:ivb-cover-nonlinear}
Let $L_\emp$ be a line segment with direction $\theta_0 \in \R/\pi\Z$. Let directions $\theta_\cov, \theta_\sma \in \R/\pi\Z$ satisfy $\theta_0 \in (\theta_{\cov},\theta_{\sma})$.
Let $A\subset \R$ be a compact set such that $L_\emp\subset \intdom \Phi_\alpha$ for every $\alpha\in A$.  Assume also that for each $\alpha\in A$ and $x\in L_\emp$, we have $\theta_\alpha(x)\in (\theta_\cov,\theta_0)$.  
Let $\cT$ be a finite tree of depth $m$. If the $N_\ii$ are recursively sufficiently large (depending on $L_\emp, \theta_{\cov},\theta_{\sma},m, A$), then we have
\begin{align}
\label{eq:ivb cover}
\Phi_\alpha(L_\emp)
\subset
\Phi_\alpha (\cup\IVB(L_\emp, \theta_{\sma}, \theta_{\cov}, \cT))
\qquad\text{for all $\alpha\in A$}.
\end{align}
\end{lem}

\begin{proof}
First, we show that if the $N_\ii$ are recursively sufficiently large depending on $L_\emp, \theta_{\cov},\theta_{\sma},m, A$, then we have that for all $\ii \in \cT$,
\begin{align}
\label{eq:ivb cover induction}
\begin{gathered}
\cup\VB(L_\ii, \theta_{|\ii|+1}, \theta_{\cov}, N_\ii)
\subset
\intdom\Phi_\alpha \qquad\text{for all $\alpha \in A$}
\\
\Phi_\alpha(L_\ii)
\subset
\Phi_\alpha (\cup\VB(L_\ii, \theta_{|\ii|+1}, \theta_{\cov}, N_\ii)) \qquad\text{for all $\alpha \in A$}
.
\end{gathered}
\end{align}
Let $k \in \{0,\ldots, m-1\}$. Suppose $(N_\ii)_{|\ii| < k}$ are sufficiently large so that \eqref{eq:ivb cover induction} holds for all $|\ii| < k$. Now fix some $\ii$ with $|\ii| = k$. We have $L_\ii \subset \intdom \Phi_\alpha$ for all $\alpha \in A$. (For $k=0$, this is a lemma hypothesis. For $k > 0$, this follows by applying the second condition of \eqref{eq:ivb cover induction} to the parent of $\ii$.) Thus, we can apply \Cref{lemma:vb-cover-nonlinear} to $L_\ii$ to see that if $N_\ii$ is sufficiently large (depending on $L_\ii, \theta_{\sma}, \theta_{\cov},m, A$), then \eqref{eq:ivb cover induction} holds for our fixed $\ii$.

Suppose the $N_\ii$ are recursively sufficiently large, as described above. Then for all $|\ii| < m$ and for all $\alpha \in A$, we have 
\[
\Phi_\alpha(L_\ii)\subset \bigcup_{\jj \text{ child of } \ii} \Phi_\alpha(L_\jj)
\]
By induction, this implies \[\Phi_\alpha(L_\emp) \subset \bigcup_{|\jj|=m} \Phi_\alpha(L_\jj) \qquad\text{for all } \alpha \in A,\] which is precisely \eqref{eq:ivb cover}.
\end{proof}

\begin{lem}[Small projections in certain directions for iterated Venetian blinds]
\label{lemma:ivb-small-nonlinear-restricted-domain}
Let $\e>0$. Let $L_\emp = \LINE(a_\emp, b_\emp)$ be a line segment with direction $\theta_0 \in \R/\pi\Z$ and with length $|L_\emp| < \e$. Let directions $\theta_\cov, \theta_\sma \in \R/\pi\Z$ satisfy $\theta_0 \in (\theta_{\cov},\theta_{\sma})$.
Let $A\subset \R$ be a compact set such that $\theta_\alpha(a_\emp)\in[\theta_0,\theta_\sma]$ for all $\alpha\in A$. If $m$ is sufficiently large and the $N_\ii$ are recursively sufficiently large (depending on $L_\emp, \theta_{\cov},\theta_{\sma},m, \e$), then we have
\[
|\Phi_\alpha(\cup\IVB(L_\emp,\theta_\sma,\theta_\cov,\cT) \cap \dom \Phi_\alpha)|\lesssim \max\left(1,\frac{\sin|\theta_0-\theta_\cov|}{\sin|\theta_\sma-\theta_\cov|}\right)\e|L_\emp|
\qquad \text{for all $\alpha\in A$.}
\]
\end{lem}

\begin{proof}
 Let $m\in\N$ be such that $\frac{1}{m}(\theta_{\sma}-\theta_0) < \e$. Let $\theta_k=\theta_0+\frac{k}{m}(\theta_\sma-\theta_0)$, as in the definition of $\IVB$.  Define
\[
A_k = \{\alpha \in A : |\theta_\alpha(a_\varnothing) - \theta_k| \leq \e\}.
\]
note that $A = \bigcup_{k=0}^m A_k$ by the choice of $m$. We record for later use that for all $\alpha \in A_k$ and all $x \in B(a_\varnothing, 2\e) \cap \dom \Phi_\alpha$,
\begin{align}
\label{eq:theta alpha x theta k}
|\theta_\alpha(x) - \theta_k| 
\leq 
|\theta_\alpha(x) - \theta_\alpha(a_\varnothing)|
+
|\theta_\alpha(a_\varnothing) - \theta_k|
\lesssim\e
.
\end{align}

\ 

In what follows, let $\ii'$ denote the parent of $\ii$ (if it exists). Now, we show that if the $N_\ii$ are recursively sufficiently large and if $\delta_\ii > 0$ are recursively sufficiently small, depending on $L_\emp, \theta_{\cov},\theta_{\sma},m, \e$, then the following holds for all $\ii \in \cT$:
\begin{gather}
\label{eq:phi-nbhd-small}
|\Phi_\alpha(\nbhd{L_\ii}{\delta_\ii}\cap \dom \Phi_\alpha)| 
\lesssim
\e |L_\ii| \text{ for all $\alpha \in A_{|\ii|}$},
\\
\label{eq:nbhd-in-S_i}
\nbhd{L_\ii}{\delta_\ii}
\text{ is a subset of } 
\begin{cases}
B(a_\varnothing, 2\e)
&\text{ if } \ii = \emp
\\
\nbhd{L_{\ii'}}{\delta_{\ii'}} &\text{ if $\ii \neq \emp$}
\end{cases}
\\
\label{eq:VB L_j subset nbhd L_j}
\cup \VB(L_\ii, \theta_{|\ii|+1}, \theta_{\cov}, N_\ii) \subset \nbhd{L_\ii}{\delta_\ii}.
\end{gather}
Let $k \in \{0,\ldots, m-1\}$. Suppose that $(N_\ii)_{|\ii| < k}$ are recursively sufficiently large and $(\delta_\ii)_{|\ii| < k}$ are recursively sufficiently small so that \eqref{eq:phi-nbhd-small}, \eqref{eq:nbhd-in-S_i}, \eqref{eq:VB L_j subset nbhd L_j} hold for all $|\ii| < k$.

Fix some $\ii$ with $|\ii| = k$.  We observe the following:
\begin{itemize}
\item If $\ii = \emp$, then $L_\ii \subset B(a_\varnothing, 2\e)$. (This follows from the lemma hypotheses.)
\item If $\ii \neq \emp$, then we have $L_\ii \subset \nbhd{L_{\ii'}}{\delta_{\ii'}} \subset B(a_\varnothing, 2\e)$. (The first inclusion follows from \eqref{eq:VB L_j subset nbhd L_j} applied to $\ii'$. The second inclusion follows from \eqref{eq:nbhd-in-S_i} applied inductively to all ancestors of $\ii$.)
\item We have $|\Phi_\alpha(L_\ii\cap \dom \Phi_\alpha)| 
\lesssim
\e |L_\ii|$ for all $\alpha \in A_{|\ii|}$.
(This follows from \eqref{eq:theta alpha x theta k} and Lemma \ref{lemma:segment-small-projection-nonlinear}.)
\end{itemize}
By the above observations, together with compactness and continuity, there exists $\delta_\ii > 0$ (depending on $L_\ii, L_{\ii'}, \delta_{\ii'}, \e$) such that \eqref{eq:phi-nbhd-small} and \eqref{eq:nbhd-in-S_i} hold for our fixed $\ii$. Then by \Cref{lemma: VBgeo}(b), by choosing $N_\ii$ sufficiently large (depending on $L_\ii,\theta_{|\ii|+1},\theta_\cov,\delta_\ii$), we have \eqref{eq:VB L_j subset nbhd L_j} for 
our fixed $\ii$.

\ 

Suppose the $N_\ii$ are recursively sufficiently large and the $\delta_\ii$ as recursively sufficiently small as described above (for all $\ii \in \cT$). Then by the nested property \eqref{eq:nbhd-in-S_i},
\begin{align}
\label{eq:Li-subset-nbhd-Lj}
L_\jj \subset \nbhd{L_\ii}{\delta_\ii} \text{ for all ancestors $\ii$ of $\jj$}.
\end{align}
To prove the claimed bound on the measure of the projection, let $\alpha\in A$ and choose $k$ such that $\alpha\in A_k$.  Then, by \eqref{eq:Li-subset-nbhd-Lj}, \eqref{eq:phi-nbhd-small}, and Lemma \ref{lemma: IVBgeo}(a) we have
\begin{align*}
|\Phi_\alpha(\cup \IVB(L_\emp,\theta_\sma,\theta_\cov,\cT)\cap \dom \Phi_\alpha)| &=|\Phi_\alpha(\bigcup_{|\ii| = m} L_\ii\cap \dom \Phi_\alpha)| \\
&\leq \sum_{|\ii|=k}|\Phi_\alpha(\nbhd{L_\ii}{\delta_\ii}\cap \dom \Phi_\alpha)| \\
&\lesssim \e\sum_{|\ii|=k}|L_\ii| \\
&\lesssim \max\left(1,\frac{\sin|\theta_0-\theta_\cov|}{\sin|\theta_\sma-\theta_\cov|}\right) \e|L_\emp|.
\end{align*}
\end{proof}

\begin{rmk}
\label{remark:mod-pi-and-direction}
In the lemmas, the Venetian blinds always rotated counterclockwise. This was only to simplify notation. We could also rotate clockwise, in which case we would have assumptions such as $\theta_0 \in (\theta_{\sma}, \theta_{\cov})$ instead.
\end{rmk}

\section{Key lemma for prescribed projections}
\label{KeyLemmaSection}
The goal of this section is to prove the following key lemma.  The conclusion is similar to our prescribed projection result, with less freedom to choose prescriptions and with error of small positive measure rather than measure zero.  It is the basis for the proof of our main theorem.
\begin{lem}[Key lemma]
\label{lemma: keylemma}
Let $y\in\{\alpha_0\}\times \R$, and let $A_\sma,A_\cov$ be disjoint compact subsets of $\R$ such that $\alpha_0\in A_\sma$ and $A_\cov$ is an interval.  Let $\gamma$ be a compact and connected subset of $\Phi_{\alpha_0}^{-1}(y)=y-\Gamma$ such that $\gamma\subset \intdom \Phi_\alpha$ for every $\alpha\in A_\cov$.  For any $\e,\delta>0$, there exists $E\subset \nbhd{\gamma}{\delta}$ satisfying:
\begin{enumerate}[label=(\alph*)]
\item\label{item:keylemma finite union segments} 
$E$ is a finite union of line segments.
\item\label{item:keylemma cover} 
$\Phi_\alpha(E)\supset \Phi_\alpha(\gamma)$ for every $\alpha\in A_{\cov}$,
\item\label{item:keylemma small} 
$|\Phi_\alpha(E \cap \dom\Phi_\alpha)| < \e$ for every $\alpha\in A_{\sma}$. 
\end{enumerate}
\end{lem}
There are two main steps.  First, we approximate the curve $\gamma$ by a polygonal curve, then we apply the iterated Venetian blind construction to each segment of the polygonal curve.  Recall that $\Gamma$ denotes the curve associated to our projections, and the pre-image of a point is given by $\Phi_{\alpha_0}^{-1}(y)=y-\Gamma$.

\begin{lem}[Polygon approximation lemma]
\label{lemma: polygon-approximation}
Let $y\in \R^2$, and let $\gamma$ be a compact and connected subset of $y-\Gamma$.  For any $\e, \delta>0$ , there exists a polygonal curve $P$ such that 
\begin{enumerate}[label=(\alph*)]
\item $P\subset\nbhd{\gamma}{\delta}$
\item For every $\alpha\in\R$, we have $\Phi_\alpha(P)\supset \Phi_\alpha(\gamma)$
\item Each segment of $P$ is tangent to $\gamma$ and has length $< \e$.
\end{enumerate}
\end{lem}
\begin{proof}
By translational symmetry, we may assume $y=0$. Let $g(t)=-f(-t)$, and let $a',b'$ be such that $\gamma=\{(t,g(t)):a'\leq t\leq b'\}$.  Let $a'=t_1<t_2<\cdots<t_{N-1}<t_N=b'$ be the partition of $[a',b']$ into $(N-1)$ equal intervals. For $i = 1, \ldots, N-1$, let $p_i$ be the intersection point of the lines tangent to $\gamma$ at the points $(t_i,g(t_i))$ and $(t_{i+1},g(t_{i+1}))$, and let $p_0=(t_1,g(t_1))$ and $p_N=(t_N,g(t_N))$.  
Let $P$ be the polygonal curve with vertices $p_0,\dots,p_N$ in that order. 

By taking $N$ sufficiently large, we have $P\subset\nbhd{\gamma}{\delta}$. By construction, each segment of $P$ is tangent to $\gamma$. Since $p_i \in [t_i,t_{i+1}]\times \R$ and the segments of $P$ have bounded slope, it follows that $|p_{i+1}- p_{i}| \lesssim t_{i+2} - t_i \lesssim \frac{1}{N}$. Thus, by taking $N$ sufficiently large, we can ensure every segment has length $< \e$.

It remains to show that for all $\alpha \in \R$, we have $\Phi_\alpha(P)\supset \Phi_\alpha(\gamma)$. For $\alpha = 0$, we have \[\Phi_0(\gamma) \subset \Phi_0(\Phi_0^{-1}(0)) = \{0\}\] and since $p_0 \in P \cap \gamma$, we have $0 = \Phi_0(p_0) \in \Phi_0(P)$, which proves the claim for $\alpha = 0$. Now, suppose $\alpha \neq 0$. Let $z=(z_1,z_2)\in\Phi_\alpha(\gamma)$, so that $z_1 = \alpha$ and there exists $x_*=(t_*,g(t_*)) \in \gamma$ such that $\Phi_\alpha(x_*) = z$. Recalling that $\Phi_\alpha^{-1}(z) = z - \Gamma$, to show $z \in \Phi_\alpha(P)$ and finish the proof of this lemma, it suffices to show $P \cap (z-\Gamma) \neq \varnothing$.  In particular, since $\gamma\subset -\Gamma$, it suffices to show
\[
P \cap (z+\gamma) \neq \varnothing.
\]
If $t_* = t_i$ for some $i$, then $x_* \in P \cap \gamma$ and we are done. Now, assume $t_*\in (t_i,t_{i+1})$ for some $i$. We break into two cases, depending on whether $x_*$ is close to an endpoint of $z+\gamma$ or not.

Case 1, $[t_i,t_{i+1}]\subset [a'+z_1,b'+z_1]$: Consider the function 
\begin{align}
\label{eq:definition-d(t)}
d(t)=g(t)-[g(t-z_1)+z_2], \qquad\text{ so that } d'(t) = g'(t) - g'(t-z_1).
\end{align}
The function $d(t)$ measures the vertical distance between $\gamma$ and $z+\gamma$ on the line $\{t\} \times \R$. Since $z_1 = \alpha \neq 0$, by the fact that $g'$ is strictly monotone, it follows that $d$ is strictly monotone too. Furthermore, $d(t_*) = 0$. It follows that $d(t_i)$ and $d(t_{i+1})$ have opposite signs, so $z+\gamma$ lies below $\gamma$ at $t_i$ and above at $t_{i+1}$, or vice versa.  But, $P$ and $\gamma$ coincide at those points, so $z+\gamma$ lies below $P$ on one side of the strip $[t_i,t_{i+1}]\times \R$ and above $P$ on the other.  Therefore, there is an intersection point somewhere in the strip. \\

Case 2, $[t_i,t_{i+1}]\not\subset [a'+z_1,b'+z_1]$: Assume without loss of generality that $t_{i+1}>b'+z_1$ and that $g'$ is strictly increasing.  Since $P$ lies below $\gamma$ everywhere (by convexity), $P$ lies below $z+\gamma$ at $t_*$.  We claim $P$ lies above $z+\gamma$ at $t_i$.  Since $P$ and $\gamma$ coincide at $t_i$, this amounts to showing $d(t_i)>0$. This in turn follows from \eqref{eq:definition-d(t)}, $z_1<t_{i+1}-b'<0$, $t_i<t_*$, and the fact that $g'$ is strictly increasing.  
\end{proof}

\begin{lem}[Key lemma, local version]
\label{lemma: keylemmalocal}
Let $A_\sma,A_\cov\subset \R$ be disjoint compact sets.  Suppose $A_\cov$ is an interval. Let $K \subset \R^2$ be a compact set such that $K \subset \dom \Phi_\alpha$ for every $\alpha\in A_\cov$. Then if $\e$ sufficiently small (depending on $A_\sma, A_\cov, K$), we have the following: Let $L \subset \operatorname{int} K$ be a line segment such that $|L| < \e$ and
\begin{align}
\label{eq:theta L small}
\theta_L\in\{\theta_\alpha(x):\alpha\in A_\sma, x \in L \cap \dom\Phi_\alpha\}.
\end{align}
Then for any $\delta>0$, there exists a set $E\subset \nbhd{L}{\delta}$ such that:
\begin{enumerate}[label=(\alph*)]
\item $\Phi_\alpha(L)\subset\Phi_\alpha(E)$ for all $\alpha\in A_\cov$.
\item $|\Phi_\alpha(E \cap \dom\Phi_\alpha)|\lesssim_{A_\sma,A_\cov,K} \e|L|$ for all $\alpha\in A_\sma$. 
\end{enumerate}
\end{lem}

\begin{proof}
First, we claim the following consequence of compactness:
There exists $\e_0 > 0$ (depending on $A_\sma,A_\cov,K$) such that if the line segment $L \subset \operatorname{int} K$ and $\delta > 0$ satisfy $|L| < \e_0$ and $\delta < \e_0$ and $\nbhd{L}{\delta} \subset K$, then 
\begin{align}
\label{eq:claim dist directions positive}
\operatorname{dist}
(\{\theta_\alpha(x):\alpha\in A_\cov, x \in \nbhd{L}{\delta}\}
,
\{\theta_\alpha(x):\alpha\in A_\sma, x \in \nbhd{L}{\delta} \cap \dom\Phi_\alpha\})
>
\e_0.
\end{align}
(The distance is taken with respect to the metric on $\R/\pi\Z$. If the second set is empty, we define the distance to be infinity.) To prove \eqref{eq:claim dist directions positive}, recall that for a fixed $x$, $\theta_\alpha(x)$ is strictly monotonic 
and hence injective. Because $A_\sma,A_\cov\subset \R$ are disjoint compact sets, for every $x \in K$,
\[
d(x) :=
\operatorname{dist}
(\{\theta_\alpha(x):\alpha\in A_\cov\}
,
\{\theta_\alpha(x):\alpha\in A_\sma, x \in \dom\Phi_\alpha\})
>
0
\]
(If the second set is empty, we define $d(x) = \infty$.) Furthermore, because $K$ is compact, $\inf_{x \in K} d(x) > 0$. Because $|\nabla \theta_\alpha(x)| \lesssim 1$, there exists $\e_1 > 0$ such that if $K' \subset K$ is any set of diameter at most $\e_1$, then
\[
\operatorname{dist}
(\{\theta_\alpha(x):\alpha\in A_\cov, x \in K'\}
,
\{\theta_\alpha(x):\alpha\in A_\sma, x \in K' \cap \dom\Phi_\alpha\})
>
\frac{1}{2} \inf_{x \in K} d(x).
\]
Choosing $\e_0 < \min(\frac{1}{3}\e_1, \frac{1}{2} \inf_{x \in K} d(x))$ proves \eqref{eq:claim dist directions positive}. \\

For the remainder of the proof, let $\e_0$ be as in \eqref{eq:claim dist directions positive} and suppose $\e < \e_0$. Suppose $L \subset \operatorname{int} K$ and $\delta > 0$ satisfy $|L| < \e$ and $\delta < \e_0$ and $\nbhd{L}{\delta} \subset K$ and \eqref{eq:theta L small}. Then \eqref{eq:claim dist directions positive}, the fact that $A_\cov$ is an interval, and the fact that $\R/\pi\Z$ is topologically a circle imply that there exist disjoint intervals $[\theta_\cov^-,\theta_\cov^+], [\theta_\sma^-,\theta_\sma^+] \subset \R/\pi\Z$ such that 
\begin{gather}
\label{eq:separation theta small theta cover}
\operatorname{dist}([\theta_\cov^-,\theta_\cov^+], [\theta_\sma^-,\theta_\sma^+])
\geq \e_0
\\
\label{eq:theta_cover- theta_cover+}
\{\theta_\alpha(x):\alpha\in A_\cov,x\in \nbhd{L_\emp}{\delta}\}\subset (\theta_\cov^-,\theta_\cov^+)
\\
\label{eq:theta_small- theta_small+}
\{\theta_\alpha(x):\alpha\in A_\sma,x\in \nbhd{L_\emp}{\delta} \cap \dom\Phi_\alpha\}\subset (\theta_\sma^-,\theta_\sma^+).
\end{gather}
By \eqref{eq:theta L small}, we have $\theta_L\in (\theta_\sma^-,\theta_\sma^+)$.

The idea of the proof is as follows.  If it were true that $\theta_L=\theta_\sma^-$, then we could proceed directly by applying the iterated Venetian blind scheme. However, this never holds, since $\theta_L \in (\theta_\sma^-,\theta_\sma^+)$. Thus, we take two steps. First, we use a (single-step, non-iterated) Venetian blind construction to rotate the segments to point in direction $\theta_\sma^-$. Then, we apply the iterated Venetian blind scheme to the resulting segments. The ``covering directions'' in the first and second step are chosen to be $(\theta_L,\theta_\cov^+)$ and $(\theta_\cov^-,\theta_\sma^-)$ respectively; this is so that their intersection is the desired interval $(\theta_\cov^-,\theta_\cov^+)$. (See \Cref{fig:intersect covering directions}.)

\begin{figure}[h]
\centering
\begin{tikzpicture}
\draw[thick] (-2,0)--(2,0);
\draw[thick] (0,-2)--(0,2);
\draw[thick] (-1.73,-1)--(1.73,1);
\draw[thick] (-1.73,1)--(1.73,-1);
\draw[thick] (-1,1.73)--(1,-1.73);

\draw [ultra thick, red] (1,0) arc [radius=1, start angle=0, end angle= 120];
\draw [ultra thick, red] (-1,0) arc [radius=1, start angle=180, end angle= 300];
\draw [ultra thick, blue] (0,.9) arc [radius=.9, start angle=90, end angle= 150];
\draw [ultra thick, blue] (0,-.9) arc [radius=.9, start angle=270, end angle= 330];
\draw [ultra thick, purple] (0,.8) arc [radius=.8, start angle=90, end angle= 120];
\draw [ultra thick, purple] (0,-.8) arc [radius=.8, start angle=270, end angle= 300];

\node[right] at (1.73,1) {$\theta_\sma^+$};
\node[right] at (2,0) {$\theta_L$};
\node[right] at (1.73,-1) {$\theta_\sma^-$};
\node[below right] at (1,-1.73) {$\theta_\cov^+$};
\node[below] at (0,-2) {$\theta_\cov^-$};
\node[left] at (-1.73,-1) {$\theta_\sma^+$};
\node[left] at (-2,0) {$\theta_L$};
\node[left] at (-1.73,1) {$\theta_\sma^-$};
\node[above left] at (-1,1.73) {$\theta_\cov^+$};
\node[above] at (0,2) {$\theta_\cov^-$};

\end{tikzpicture}
\caption{The angle intervals $(\theta_L,\theta_\cov^+)$ in red and $(\theta_\cov^-,\theta_\sma^-)$ in blue, and their intersection $(\theta_\cov^-,\theta_\cov^+)$ in purple.}
\label{fig:intersect covering directions}
\end{figure}

For the first step, let
\[
\mathcal{V}=\VB(L, \theta_\sma^-,\theta_\cov^+,N),
\]
where $N$ is sufficiently large so that $\cup\mathcal{V} \subset \nbhd{L}{\delta/2}$ and so that 
the conclusions of \Cref{lemma:vb-cover-nonlinear} hold for $\VB(L, \theta_\sma^-,\theta_\cov^+,N)$ and $A_{\cov}$. Note that $\theta_L  \in (\theta_{\sma}^-, \theta_{\cov}^+)$, so this Venetian blind is constructed in the clockwise direction; see \Cref{remark:mod-pi-and-direction}. Note also that \eqref{eq:theta_cover- theta_cover+} holds and that $(\theta_\cov^-,\theta_\cov^+) \subset (\theta_L,\theta_\cov^+)$, so the hypotheses of Lemma 2.5 are satisfied.  By \Cref{lemma:vb-cover-nonlinear}, we have 
\begin{align}
\label{reorient}
\Phi_\alpha(\cup \mathcal{V})&\supset \Phi_\alpha(L)
\qquad\text{ for every } \alpha \in A_\cov
.
\end{align}

Next, for each $L'\in \mathcal{V}$, apply the iterated Venetian blind process to construct
\[
\mathcal{V}_{L'}:=\IVB(L',\theta_\sma^+,\theta_\cov^-,\cT_{L'}),
\]
with the tree $\cT_{L'}$ chosen so that the conclusions of \Cref{lemma:ivb-cover-nonlinear} and \Cref{lemma:ivb-small-nonlinear-restricted-domain} are satisfied and so that $\cup\mathcal{V}_{L'}\subset\nbhd{L}{\delta}$.  (This iterated Venetian blind is constructed in the counter-clockwise direction.) Finally, define
\[
E:=\bigcup_{L'\in\mathcal{V}} (\cup\mathcal{V}_{L'}).
\]
We claim this set satisfies the conclusion of the lemma.  \\

If $\alpha\in A_\cov$, it follows that
$
\theta_\alpha(x)\in (\theta_\cov^-,\theta_\cov^+) \subset (\theta_\cov^-,\theta_\sma^-)
$
for every $x\in \nbhd{L}{\delta}$.
In particular, if $\alpha \in A_{\cov}$ and $L' \in \cV$, then we have $\theta_\alpha(x)\in [\theta_\cov^-,\theta_\sma^-]$ for every $x \in L'$, so Lemma \ref{lemma:ivb-cover-nonlinear} implies
\[
\Phi_\alpha(\cup \mathcal{V}_L)\supset \Phi_\alpha(L).
\]  
Taking the union over $L'\in\mathcal{V}$ and using (\ref{reorient}), we conclude $\Phi_\alpha(E)\supset \Phi_\alpha(L)$.  

If $\alpha\in A_\sma$ and $L' \in \mathcal{V}$, then
\[
\theta_\alpha(x)\in (\theta_\sma^-,\theta_\sma^+)
\qquad\text{for every $x\in L'$},
\]
hence by Lemma \ref{lemma:ivb-small-nonlinear-restricted-domain} we have
\[
|\Phi_\alpha(\cup \mathcal{V}_{L'})|
\lesssim 
\max\left(1,\frac{\sin|\theta_\sma^--\theta_\cov^-|}{\sin|\theta_\sma^+-\theta_\cov^-|}\right)
\e|L'|.
\]
Summing over $L'\in \mathcal{V}$ and using \Cref{lemma: VBgeo} and \eqref{eq:separation theta small theta cover} gives
\begin{align*}
|\Phi_\alpha(E)|
&\lesssim 
\max\left(1,\frac{\sin|\theta_\sma^--\theta_\cov^-|}{\sin|\theta_\sma^+-\theta_\cov^-|}\right)
\e
\sum_{L' \in \cV} 
|L'|
\\
&=
\max\left(1,\frac{\sin|\theta_\sma^--\theta_\cov^-|}{\sin|\theta_\sma^+-\theta_\cov^-|}\right)
\frac{\sin|\theta_L - \theta_{\cov}^+|}{\sin|\theta_{\sma}^- - \theta_{\cov}^+|}
\e|L|
\\
&\lesssim_{\e_0} \e |L|
\end{align*}
as desired.
\end{proof}

\begin{rmk}
\Cref{lemma: keylemmalocal} is a local nonlinear analogue of \cite[Theorem 2.4]{Fal86_paper}. Falconer's construction uses two iterated Venetian blinds, together with some affine transformations. Our approach uses only one iterated Venetian blind, and it avoids affine transformations, which could cause complications for nonlinear curves. 
\end{rmk}
We are now ready to prove the key lemma stated at the beginning of the section.  The idea is to use the polygon approximation lemma (\Cref{lemma: polygon-approximation}) to reduce to the local version.  
\begin{proof}[Proof of Lemma \ref{lemma: keylemma}]
Let $y,\alpha_0,\gamma,A_\sma,A_\cov,\e,\delta$ be as in the statement of Lemma \ref{lemma: keylemma}, and recall that $\Phi_{\alpha_0}^{-1}(y)=y-\Gamma$. By making $\delta$ smaller if necessary, we may assume that $\nbhd{\gamma}{\delta} \subset \intdom \Phi_\alpha$ for every $\alpha \in A_\cov$. Let $K$ be the closure of $\nbhd{\gamma}{\delta}$. By making $\e$ smaller if necessary (depending on $A_\sma,A_\cov,K$), we may assume that \Cref{lemma: keylemmalocal} can be applied with $A_\sma,A_\cov,K,\e$.

We must construct $E\subset\nbhd{\gamma}{\delta}$ satisfying properties (a)--(c) in the statement of the lemma.  The first step in the construction is to approximate the curve by a polygon.  Indeed, by applying Lemma \ref{lemma: polygon-approximation} with $\e, \delta$ and translating, we may obtain a polygonal curve $P = \bigcup_{i=1}^N S_i \subset \nbhd{\gamma} {\delta/2}$, where each $S_i$ is a line segment with $|S_i|<\e$, and such that for every $\alpha\in \R$, we have $\Phi_\alpha(P)\supset \Phi_\alpha(\gamma)$. 

For each $i$, we want to apply Lemma \ref{lemma: keylemmalocal} with $S_i$ as the initial segment.  To do this, we must check the hypotheses of that lemma are met.  By assumption, $A_\cov$ and $A_\sma$ are compact sets, and $A_\cov$ is an interval. By construction, each $S_i$ satisfies $|S_i|<\e$. If $x_i\in \gamma$ is the point at which the segment $S_i$ is tangent, then $\theta_{\alpha_0}(x_i)$ is the direction of the initial segment $S_i$.  Because $\alpha_0\in A_\sma$ by the hypotheses of Lemma \ref{lemma: keylemma}, the assumption \eqref{eq:theta L small} is satisfied.  The only other hypothesis of Lemma \ref{lemma: keylemmalocal} is that $S_i\subset \operatorname{int} K$ for all $\alpha\in A_\cov$. This is true by definition of $K$. 

Now that we have verified the hypotheses, we apply Lemma \ref{lemma: keylemmalocal} to obtain sets $E_i\subset \nbhd{S_i}{\delta/2}$ corresponding to each $S_i$ satisfying the conclusion of Lemma \ref{lemma: keylemmalocal}. 
Let $E=\bigcup_i E_i$.  We have $E\subset \nbhd{\gamma}{\delta}$ by construction.  Combining Lemmas \ref{lemma: polygon-approximation} and \ref{lemma: keylemmalocal}, for $\alpha\in A_\cov$ we have
\[
\Phi_\alpha(E)\supset \Phi_\alpha(P)\supset \Phi_\alpha(\gamma),
\]
which establishes part \ref{item:keylemma cover}.  To prove part \ref{item:keylemma small}, we use Lemma \ref{lemma: keylemmalocal} to get
\begin{align*}
|\Phi_\alpha(E \cap \dom\Phi_\alpha)|&\leq \sum_i |\Phi_\alpha(E_i\cap \dom\Phi_\alpha)| 
\lesssim_{A_\sma,A_\cov,K} \sum_{i}\e|S_i|
\lesssim \e.
\end{align*}
This proves (c) with an extra constant factor, which we can remove by making $\e$ smaller.
\end{proof}

\section{Proof of main theorem}
\label{ProofSection}
Recall the definition of $\Phi(E)$ given in section \ref{notation}: 
\[
\Phi(E)=\bigcup_\alpha \Phi_\alpha(E\cap\dom\Phi_\alpha)=E+\Gamma.
\]
With this notation, our main theorem can be restated as follows:
\begin{thm*}[Rephrasing of Theorem \ref{maintheorem}]
Let $\newA\subset \R^2$ be measurable.  There exists a set $E\subset \R^2$ such that 
\begin{enumerate}[label=(\alph*)]
\item $\Phi(E)\supset \newA$,
\item $|\Phi(E)\setminus \newA|=0$.
\end{enumerate}
\end{thm*}
By Fubini, this is equivalent to Theorem \ref{prescribedprojection}.

The goal of this section is to prove this theorem.  Most of the geometric content of the proof is already contained in our key lemma (Lemma \ref{lemma: keylemma}).  What remains is a fairly general framework first used in the linear setting by \cite{Fal86_paper}, which has been adapted to our present setting.
\begin{lem}
\label{difference}
If $D$ is an open disk in the plane with closure $\overline{D}$, then $|\Phi(\overline{D})\setminus \Phi(D)|=0$.
\end{lem}
\begin{proof}
If $E$ is Borel, then $\Phi(E)=\Gamma+E$ is the image of the Borel set $\Gamma\times E$ under the continuous map $(x,y)\mapsto x+y$, hence it is measurable.  We claim the set $\Phi(\overline{D})\setminus \Phi(D)$ intersects any vertical line in at most two points; the statement of the lemma then follows from measurability of the set and Fubini.  To prove the claim, fix $\alpha$ and observe 
\[
(\Phi(\overline{D})\setminus \Phi(D))\cap(\{\alpha\}\times\R)=\Phi_\alpha(\overline{D})\setminus\Phi_\alpha(D)
\subset
\overline{\Phi_\alpha(D)}\setminus\Phi_\alpha(D)
.
\]
Since $D$ is connected and $\Phi_\alpha$ is continuous, $\Phi_\alpha(D)$ is a connected subset of the vertical line $\{\alpha\}\times\R$. Thus, $\overline{\Phi_\alpha(D)}\setminus\Phi_\alpha(D)$ contains at most two points, as desired.
\end{proof}

\begin{dfn}
\label{definition: system-of-disks}
Let $\newA\subset \R^2$ be measurable.  Suppose for each $m\in\N$ we have a measurable set $\newA_m\supset \newA$ satisfying
\begin{enumerate}[label=(\roman*)]
\item $|\newA_m\setminus \newA|\to 0$ as $m\to\infty$.
\end{enumerate}
We refer to $\{\newA_m:m\in\N\}$ as a \textbf{system of approximations} to $\newA$.  Suppose also for each $m$ that $\cD_m$ is a countable family of open disks in the plane satisfying the following:
\begin{enumerate}[label=(\roman*)]\addtocounter{enumi}{1}
\item $\Phi(\cup \cD_m)\supset \newA_m$,
\item $|\Phi(\cup \cD_m)\setminus \newA_m|\to 0$ as $m\to\infty$,
\item For any $D\in\cD_{m-1}$, if $y\in \Phi(D)\cap \newA_m$ then there exists $D'\in\cD_m$ with $D'\subset D$ and $y\in\Phi(D')$.
\end{enumerate}
We call $\{\cD_m:m\in\N\}$ a \textbf{system of disks} corresponding to $\{\newA_m:m\in\N\}$.
\end{dfn}

\begin{lem}[Systems of disks produce efficient coverings]
\label{disks}
Let $\newA\subset \R^2$ be measurable.  Let $\{\newA_m:m\in\N\}$ be a system of approximations and $\{\cD_m:m\in\N\}$ a corresponding system of disks.  Let $E=\bigcap_m \bigcup_{D\in\cD_m}\overline{D}$.  Then, we have $\Phi(E)\supset \newA$ and $|\Phi(E)\setminus \newA|=0$.
\end{lem}
\begin{proof}
First we show $\Phi(E)\supset \newA$.  Let $y\in \newA$; by condition (ii) of Definition \ref{definition: system-of-disks}, $\newA\subset \newA_1\subset \Phi(E_1)$, so there exists $D_1\in \mathcal{D}_1$ with $y\in \Phi(D_1)$.  Since $\newA\subset \newA_2$ also, we have $y\in \Phi(D_1)\cap \newA_2$.  By condition (iv) of Definition \ref{definition: system-of-disks}, there exists $D_2\in \mathcal{D}_2$ with $D_2\subset D_1$ and $y\in \Phi(D_2)$.  But, since $\newA\subset \newA_3$ as well, we have $y\in \Phi(D_2)\cap \newA_3$.  Continuing in this way, we obtain a nested sequence of disks $D_1\supset D_2\supset D_3\supset\cdots$ such that $y\in \Phi(D_m)$ for each $m$.  Let $x_m\in D_m$ be such that $y\in \Phi(x_m)$.  By compactness, there exists $x\in \bigcap_m \overline{D}_m$ such that $y\in \Phi(x)$, and $\bigcap_m \overline{D}_m\subset E$, so $y\in \Phi(E)$.  

To show $|\Phi(E)\setminus \newA|=0$, let $E_m = \cup \cD_m$ and $\widetilde E_m = \bigcup_{D\in\cD_m}\overline{D}$. Observe that for any fixed $m$ we have the following decomposition:
\[
|\Phi(E)\setminus \newA|
\leq
|\Phi(\widetilde E_m)\setminus \newA|
\leq 
|\Phi(\widetilde E_m)\setminus \Phi(E_m)|
+
|\Phi(E_m)\setminus \newA_m|
+
|\newA_m\setminus \newA|.
\]
The first term on the right is zero by Lemma \ref{difference}, and the second and third terms tend to zero as $m\to\infty$ by (iii) and (i) of \Cref{definition: system-of-disks}, respectively.
\end{proof}

\begin{lem}
\label{differenceset}
Recall that $\Gamma = \{(t,f(t)) : t \in [0,1]\}$. Let $[a', b'] \subset [0,1]$ and $\widetilde\Gamma = \{(t,f(t)) : t \in [a',b']\}$.  Also, let $\Gamma^\circ=\{(t,f(t)):t\in (0,1)\}$ denote the relative interior of the curve $\Gamma$, and define $\widetilde\Gamma^\circ$ similarly.  
\begin{enumerate}[label=(\alph*)]
    \item The set $\Gamma-\widetilde\Gamma$ has non-empty interior.  More precisely,
    \begin{align}\label{eq:int gamma-gamma}
    \textup{int}(\Gamma-\widetilde\Gamma)= (\Gamma^\circ-\widetilde\Gamma^\circ)\setminus \{0\}.
    \end{align}
    \item The boundary of $\Gamma-\widetilde\Gamma$ has measure zero.
    \item For any interval $(c,d) \subset \R$, the set $((c,d)\times \R)\cap (\Gamma-\widetilde\Gamma)$ is contained in a horizontal strip of width $\lesssim \max\{|c|,|d|\}$, where the implied constant only depends on $f$.
\end{enumerate}

\end{lem}

\begin{proof}
Without loss of generality, assume $f'$ is strictly increasing. By definition of $\Gamma-\widetilde\Gamma$, we have
\begin{align}
\label{eq:gamma-gamma expression}
\Gamma-\widetilde\Gamma
&=
\{(s, f(t+s)-f(t)) : t \in [a',b'], t+s \in [0,1]\}
=
\bigcup_{s \in [-b', 1-a']}
\{s\} \times I_s
\end{align}
where
\begin{align*}
I_s = \{f(t+s)-f(t):t\in [\max\{-s, a'\},\min\{1-s,b'\}]\}
\end{align*}
Observe that for $s \in \{-b',0,1-a'\}$, the set $I_s$ consists of a single point, whereas for $s \in (-b',0) \cup (0,1-a')$
the set $I_s$ is a nondegenerate interval (this is because $f'$ is strictly increasing, so the function $t \mapsto f(t+s)-f(t)$ is strictly monotone on the nondegenerate interval $[\max\{-s, a'\},\min\{1-s,b'\}]$). A similar argument shows that
\begin{align}
\label{eq:gammacirc-gammacirc expression}
\Gamma^\circ-\widetilde\Gamma^\circ
&=
\{(s, f(t+s)-f(t)) : t \in (a',b'), t+s \in (0,1)\}
=
\{0\} \cup
\bigcup_{s \in (-b',0) \cup (0, 1-a')}
\{s\} \times \operatorname{int}I_s
\end{align}
Comparing \eqref{eq:gamma-gamma expression} and \eqref{eq:gammacirc-gammacirc expression} gives \eqref{eq:int gamma-gamma} and proves (a). By \eqref{eq:gamma-gamma expression}, the boundary of $\Gamma-\widetilde\Gamma$ is 
\[
\bigcup_{s \in [-b', 1-a']}
\{s\} \times \partial I_s
\]
which has measure zero, which proves (b). Finally, if $s \in (c,d)$, then by the mean value theorem,
\begin{align*}
\max\{|f(t+s)-f(t)| : t\in [\max\{-s, a'\},\min\{1-s,b'\}]\}
\leq \|f'\|_\infty |s|
\leq \|f'\|_\infty \max\{|c|, |d|\}
\end{align*}
which proves (c).
\end{proof}

\begin{lem}[Topological base lemma]
\label{base}
Let $D,V\subset\R^2$ be open sets such that $\Phi(D)\supset V$, and assume $D\subset \intdom \Phi_\alpha$ for any $\alpha$ such that the line $\{\alpha\}\times \R$ passes through $\overline{V}$.  There is a base $\mathcal{B}$ for the topology of $V$ such that for any $U\in\mathcal{B}$ and $\e>0$, there exists a set $E\subset D$ which satisfies 
\begin{enumerate}[label=(\alph*)]
\item $E$ is a finite union of open disks,
\item \label{item:baselemma cover} $U\subset \Phi(E)$,
\item \label{item:baselemma small} $|\Phi(E)\setminus U|<\e$.
\end{enumerate}
\end{lem}
\begin{proof} 
If $y\in \R^2$, $\gamma$ is a compact and connected subset of $y-\Gamma$, and $I$ is an interval, we define
\[
U_{y,\gamma,I}=\operatorname{int}[(I\times \R)\cap \Phi(\gamma)].
\]
Let $\mathcal{B}$ be the set of all $U_{y,\gamma,I}$ generated by parameters satisfying:
\begin{itemize}
    \item $y = (y_1, y_2) \in V$,
    \item $\gamma$ is a compact and connected subset of $y-\Gamma$ consisting of more than one point, and $\gamma \subset D$,
    \item $I$ is an interval of the form $(y_1+\e',y_1+2\e')$, where $\e'>0$ is sufficiently small so that for all $\alpha \in I, (\{\alpha\} \times \R) \cap V \neq \varnothing$,
    \item $U_{y,\gamma,I}\subset V$.
\end{itemize}
We first show that $\mathcal{B}$ is a topological base for $V$. Let $y, \gamma, I$ be as above. Then $\gamma = y - \widetilde\Gamma$ where $\widetilde\Gamma \subset \Gamma$ is as in \Cref{differenceset}. Thus, $\Phi(\gamma)=y+\Gamma-\widetilde\Gamma$. By Lemma \ref{differenceset}, $U_{y,\gamma,I}$ is nonempty, provided $\e'$ is sufficiently small (depending only on $\gamma$); furthermore, $U_{y,\gamma,I}$ is contained in a rectangle of dimensions $\e'$ and $C\e'$, where $C$ only depends on $f$. Thus, to show that $\mathcal{B}$ is a base, it is enough to show that 
\begin{multline}
\label{eq:base condition}
\text{for any $z\in V$,}
\\
\text{there exist $y \in V, \gamma \subset D, I = (y_1 + \e', y_1 + 2\e')$ with $\e'$ arbitrarily small such that $z \in U_{y,\gamma,I}$.}
\end{multline}

To prove this, let $z \in V$. Then $z \in V \subset \Phi(D) = D + \Gamma$. Since $D$ is open, $D + \Gamma = D + \Gamma^\circ$, where $\Gamma^\circ$ is the relative interior of $\Gamma$. Thus, there exists $x \in D$ such that $z \in x + \Gamma^\circ$.
Given $\e'>0$, choose $y\in x+\Gamma^\circ$ such that $y\in V$ and such that $\e'<z_1-y_1<2\e'$. (This is possible if $\e'$ is sufficiently small.) 
Since $x \in D \cap (y-\Gamma^\circ)$, we can let $\gamma \subset (y - \Gamma)$ be a compact and connected subset satisfying $x \in \gamma^\circ$ and $\gamma\subset D$.  Since $z \in x + \Gamma^\circ \subset \Gamma^\circ + \gamma^\circ$, \Cref{differenceset} implies that $z$ is an interior point of $(I\times \R)\cap \Phi(\gamma)$, so $z\in U_{y,\gamma,I}$.  This proves \eqref{eq:base condition} and shows that $\mathcal{B}$ is a base.


  

To construct the set $E$, fix $\e>0$ and fix $U=U_{y,\gamma,I}\in\mathcal{B}$.  Let $\alpha_0$ be the first coordinate of $y$, let $I=(\alpha_0+\e',\alpha_0+2\e')$, and let $I'=(\alpha_0+\e'-\e,\alpha_0+2\e'+\e)$ be the $\e$-neighborhood of $I$ (taking a smaller $\e$ if necessary, we may assume $\e<\e'$).  Let $\delta \in (0,1)$ to be determined later, and apply \Cref{lemma: keylemma} with $A_{\text{cover}}=\overline{I}$ and $A_{\text{small}}=[-N,N]\setminus I'$, and $\delta$, where $N$ is large enough so that $\Phi(\nbhd{\gamma}{1}) \subset [-N,N] \times \R$. \Cref{lemma: keylemma} gives us a set $\widetilde E\subset\nbhd{\gamma}{\delta}\subset D$ such that $\Phi_\alpha(\widetilde E \cap \dom\Phi_\alpha)\supset \Phi_\alpha(y-\Gamma)$ for $\alpha\in I$ and $|\Phi_\alpha(\widetilde E)| < \e$ for $\alpha\notin I'$ (in justifying the assumptions of Lemma \ref{lemma: keylemma}, we use the fact that $\gamma\subset D$, and that $D\subset \intdom \Phi_\alpha$ for any $\alpha$ such that $\{\alpha\}\times \R$ intersects $\overline{V}$, and hence for any $\alpha\in \overline{I}=A_\cov$). 
Note that the set $\widetilde{E}$ is a union of line segments; to get condition (a), we can find a set $E \supset \widetilde E$ which is a finite union of open disks such that the properties of the previous sentence hold with $E$ in place of $\widetilde E$. (Just make sure the disks are sufficiently small.)

To show property \ref{item:baselemma cover}, let $x\in U$.  Then $x=(\alpha,x_2)$ with $\alpha\in I\subset A_{\text{cover}}$, and $x\in\Phi_\alpha(\gamma)$ by definition of $U_{y,\gamma,I}$, so $x\in \Phi_\alpha(E)$ by Lemma \ref{lemma: keylemma}.  To prove property \ref{item:baselemma small}, we use the fact that $U \subset A_\cov \times \R$ and $\Phi(E) \subset [-N,N] \times \R$  to obtain
\begin{align}
\label{eq:Phi E minus U}
|\Phi(E)\setminus U|
&=
|(\Phi(E)\setminus U) \cap (A_\cov \times \R)|
+
|\Phi(E) \cap (A_\sma \times \R)|
+
|\Phi(E) \cap ((I'\setminus I) \times \R|
\end{align}
We now bound each of these three terms separately.

For the first term in \eqref{eq:Phi E minus U},  we use the fact that $E\subset \nbhd{\gamma}{\delta}$, the fact that $|U| = |\overline U|$ (which follows from \Cref{differenceset}(b)), and the definition of $U$ to obtain
\begin{align*}
|(\Phi(E)\setminus U) \cap (A_\cov \times \R)|
&\leq
|(\Phi(\nbhd{\gamma}{\delta})\setminus U) \cap (A_\cov \times \R)|
\\
&=
|(\Phi(\nbhd{\gamma}{\delta})\setminus \overline U) \cap (A_\cov \times \R)|
\\
&=
|(\Phi(\nbhd{\gamma}{\delta})\setminus \Phi(\gamma)) \cap (A_\cov \times \R)|.
\end{align*}
Furthermore, $\Phi(\nbhd{\gamma}{\delta}) \subset \nbhd{\Phi(\gamma)}{C\delta}$ for some $C$ depending only on the function $f$, and $\Phi(\gamma)$ is compact. Thus, by choosing $\delta$ small enough, this term is less than $\e$.

To bound the second term in \eqref{eq:Phi E minus U}, we observe that if $\alpha\in A_{\sma}$ then $|\Phi_\alpha(E \cap \dom\Phi_\alpha)|<\e$. Hence, by Fubini, \begin{align*}
|\Phi(E) \cap (A_\sma \times \R)| \leq |A_\sma| \e \leq 2N\e.    
\end{align*} 
(Recall that $N$ was chosen so that $\Phi(\nbhd{\gamma}{1}) \subset [-N,N] \times \R$. It is important that $N$ does not depend on $\e$.)

Finally, to bound the third term in \eqref{eq:Phi E minus U}, we observe that the set $\Phi(\gamma)=\gamma+\Gamma$ has the property that its intersection with every vertical line is an interval of length $\lesssim 1$. (The implied constant only depends on $f$.) Thus, the same is true for $\Phi(\nbhd{\gamma}{\delta})$. Hence, by Fubini,
\begin{align*}
|\Phi(E) \cap ((I'\setminus I) \times \R|
\leq
|\Phi(\nbhd{\gamma}{\delta}) \cap ((I'\setminus I) \times \R)|
\leq
C|I'\setminus I|
=
2C\e.
\end{align*}

This shows that $|\Phi(E) \setminus U| \lesssim \e$, provided $\delta$ is chosen small enough.
\end{proof}

We have shown that systems of disks produce efficient coverings, and that we may construct a topological base with desired projection properties.  All that remains to prove our main theorem is to show that such a base can be used to construct a system of disks.  We prove this here.

\begin{proof}[Proof of Theorem \ref{maintheorem}]
Without loss of generality, we may assume the diameter of the set $W$ is at most $\frac{1}{10^{10}}$ (otherwise, the compact set $\newA$ could be partitioned into finitely many pieces, each of which satisfies that assumption, and by covering each piece efficiently we would get an efficient covering of the entire set).  By Lemma \ref{disks}, it suffices to construct a system of approximations $\{\newA_m\}_{m=1}^\infty$ and disks $\{\mathcal{D}_m\}_{m=1}^\infty$.  We do this as follows.  By standard properties of measure we may choose open sets $\newA_m\supset \newA$ such that $\newA_1 \supset \newA_2 \supset \newA_3 \supset \cdots$ and $|\newA_m\setminus \newA|\to 0$ as $m\to\infty$; we also assume each set $\newA_m$ has diameter at most $\frac{1}{10^9}$.  Let $T_{\text{out}}=(c,d)\times \R$ be a vertical strip of width $\frac{1}{10^8}$ containing $\newA$ and all $\newA_m$, and let $T_{\text{in}}=(d-1,c)\times \R$.  Notice that for any $\alpha\in (c,d)$, we have $T_{\text{in}}\subset \intdom \Phi_\alpha$.   We will recursively construct the families of disks $\{\cD_m\}_{m=0}^\infty$ so that for all $m \geq 1$:
\begin{enumerate}
\item 
$\cD_m$ is a countable collection of open disks contained in $T_{\text{in}}$
\item
$\Phi(\cup \cD_m)\supset \newA_m$,
\item $|\Phi(\cup \cD_m)\setminus \newA_m| \leq 2^{-m}$,
\item For any $D\in\cD_{m-1}$, if $y\in \Phi(D)\cap \newA_m$ then there exists $D'\in\cD_m$ with $D'\subset D$ and $y\in\Phi(D')$.
\end{enumerate}
To start, let $\mathcal{D}_0$ be any countable collection of open disks whose union is $T_{\text{in}}$. Next, suppose $\mathcal{D}_1,\dots,\mathcal{D}_{m-1}$ are countable families of open disks which satisfy the properties above, and define $\mathcal{D}_m$ as follows. Let $\{D_r\}_{r\in\N}$ be an enumeration of the disks in $\mathcal{D}_{m-1}$. For each $r$, we apply \Cref{base} with $V = V_r:=\newA_m\cap \Phi(D_r)$ and $D=D_r$ (note that the hypotheses are satisfied because $V_r\subset \newA_m\subset T_{\text{out}}$ and $D_r\subset T_{\text{in}}$, and $\Phi(D_r)\supset V_r$) to get a countable collection of open sets $\{U_{r,s}\}_{s\in\N}$ such that
$V_r = \bigcup_s U_{r,s}$, and for each $s$, there exists a set $E_{r,s} \subset D_r$ which is a finite union of open disks such that $U_{r,s} \subset \Phi(E_{r,s})$ and $|\Phi(E_{r,s})\setminus U_{r,s}| < 2^{-(m+r+s)}$. Let $\mathcal{D}_m$ be the countable collection of open disks making up the sets $E_{r,s}$ (each of which is a finite union of open disks), so that $\cup \mathcal{D}_m = \bigcup_{r,s} E_{r,s}$.

Now we verify properties (1)--(4) hold for $\cD_m$.  Property (1) holds by construction.  To prove property (2), we observe that $W_m \subset W_{m-1} \subset \Phi(\bigcup_r D_r)$, so
\begin{align*}
W_m 
=
\bigcup_{r} (W_m \cap \Phi(D_r))
=
\bigcup_r V_r
=
\bigcup_{r,s} U_{r,s}
\subset
\Phi\left(\bigcup_{r,s} E_{r,s}\right)
=
\Phi(\cup \cD_m).
\end{align*}
To prove property (3), since $\cup \mathcal{D}_m=\bigcup_{r,s} E_{r,s}$, and for each fixed $r,s$ we have $U_{r,s}\subset \newA_m$, it follows that
\[
|\Phi(\cup \mathcal{D}_m)\setminus \newA_m|\leq \sum_{r,s} |\Phi(E_{r,s})\setminus U_{r,s}|\leq \sum_{r,s}2^{-(m+r+s)}=2^{-m}.
\] 
Finally, to prove property (4), let $D_r \in \cD_{m-1}$ and $y\in \Phi(D_r)\cap \newA_m$. Let $r,s$ be such that $y\in U_{r,s}$ (this can be done since $\{U_{r,s}\}_{s\in\N}$ is a topological base for $\Phi(D_r)\cap \newA_m$).  By \Cref{base}(a), it follows that $y\in \Phi(E_{r,s})$, and hence $y\in \Phi(D')$ for one of the disks $D'\in\mathcal{D}_m$ making up $E_{r,s}$.  Since $E_{r,s}\subset D_r$, we have $D'\subset D_r$.

Having verified (1)--(4), we can apply \Cref{disks} to finish the proof of the theorem.
\end{proof}

\appendix

\section{Appendix: Proof of parabola example in Remark \ref{general_classes}}
\label{parabola}
Let $\ell_c$ denote the vertical line through the point $(c,0)$ (see \Cref{pararaw2}), and observe that the intersection of this line with the graph of 
$P(a,b)=\{x\in\R^2:x_2=x_1^2+ax_1+b\}$
can be expressed as follows:
$$l_c \cap P(a,b) =  (c, c^2 + ac +b) = (c, c^2 + (a,b)\cdot(c,1)).$$

By the Prescribed Projection Theorem, there exists a planar set $E$  so that, for almost every $c\in \mathbb{R}$, 
$$E\cdot (c,1) \supset \left(\newA \cap l_c\right)_y  - c^2 \,\,\,\,\,\, \text{ and }  \,\,\,\,\,\, |E\cdot (c,1) |      = |(\newA \cap \ell_c)_y -  c^2|,$$
where we use $\left(\newA \cap l_c\right)_y $ to denote the orthogonal projection of  $\newA \cap l_c$ to the $y$-axis, 
and $|\cdot|$ denotes the $1$-dimensional Lebesgue measure.  
Let $P = \bigcup_{(a,b) \in E} P(a,b)$.  Then $P$ covers $\newA$ and, by Fubini's theorem, the sets $P$ and $\newA$ have the same $2$-dimensional Lebesgue measure. 
\begin{figure}[hhh]
\label{pararaw2}
\centering
\includegraphics[scale=.3]{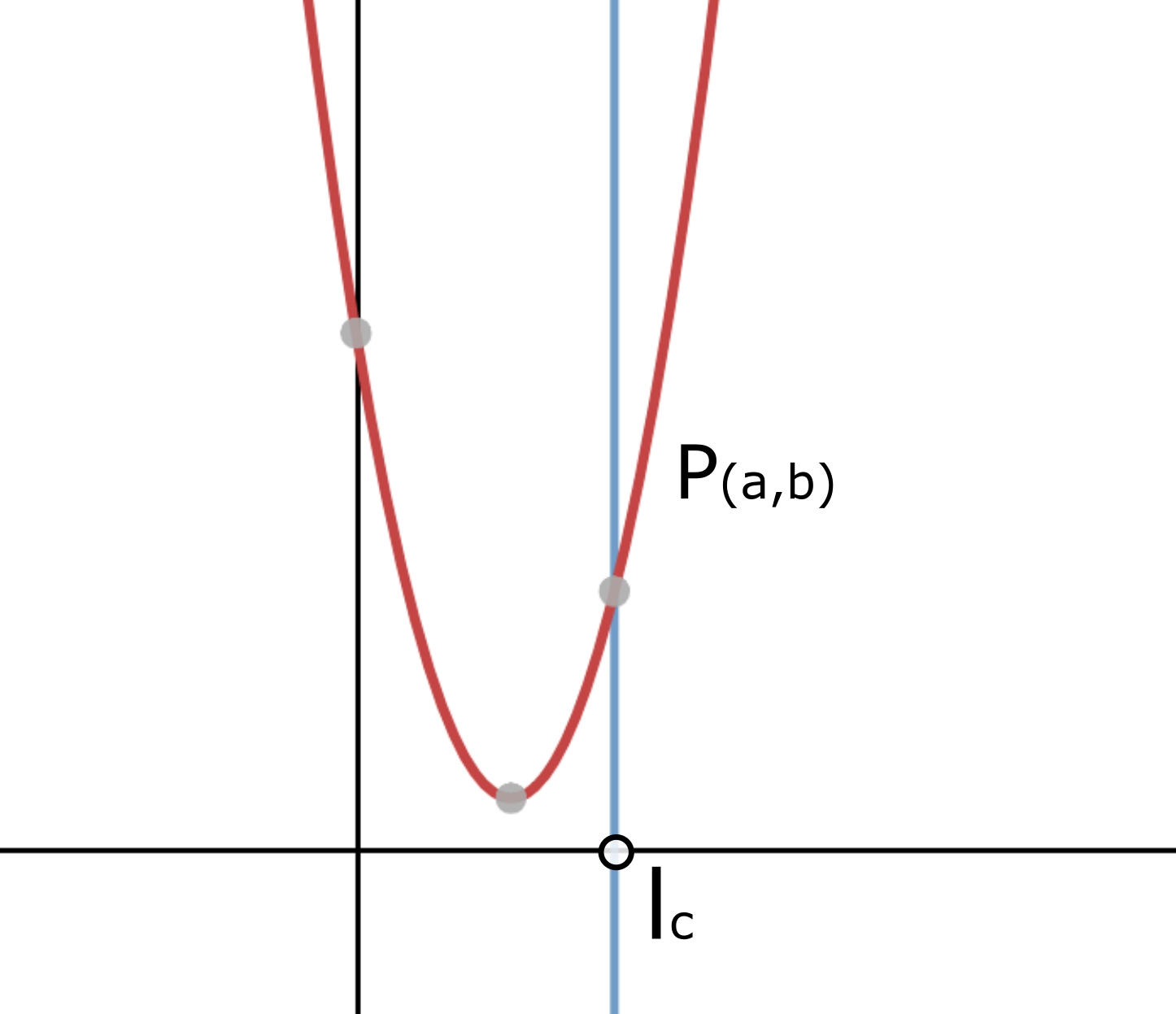}
\caption{Intersection of parabola with vertical line}
\end{figure}

\section{Appendix: Duality and the equivalence of Davies' theorem and the prescribed projection theorem}
\label{duality}
We demonstrate the equivalence of Davies' efficient covering theorem and the prescribed projection theorem in the linear setting.  
The basic idea is to identify lines with points in such a way that the projection of a set of points in a planar set $E$ is geometrically similar to the intersection of the collection of lines parametrized by $E$ with some fixed line.  
We follow the notation and ideas outlined in \cite[Section 7.3]{Fal86_book} and fill in some details. 
\\

To this end, let $L(a,b)$ denote the set of points on the line $y= a+bx$ for $(a,b) \in \R^2$. 
For $E\subset \R^2$, let $L(E) \subset \R^2$ be the line set defined by 
\begin{equation}\label{line_set}
L(E) = \bigcup_{(a,b)\in E}L(a,b).
\end{equation}
For $c\in \R$, let $L_c$ denote the vertical line $x=c$.  
Observe that 
$$L(a,b) \cap L_c = (c, a+bc) =  (c, (a,b) \cdot (1,c) ), $$ 
and that 
$$(a,b) \cdot (1,c)/ \sqrt{1+c^2} =  \pi_\theta(a,b),$$
where $c = \tan{\theta}$ and $\pi_\theta$ denotes the orthogonal projection, $\pi_\theta(a,b) = (a,b) \cdot (\cos{\theta}, \sin{\theta})$. 
Now, setting $n_c = \sqrt{1+c^2}$, 
\begin{equation}\label{equiv}
L(E) \cap L_c =  (c, n_c \cdot \pi_\theta(E)).
\end{equation}
In other words, $L(E) \cap L_c$ is geometrically similar to $\pi_\theta(E)$.  
\\

To see that the prescribed projection theorem implies Davies' efficient covering theorem, 
let $A\subset \R^2$, and let $A_c$ denote the slice $A_c = A\cap L_c.$
Set 
\begin{equation}\label{Atheta}
A_\theta = \frac{1}{n_c} A_c,
\end{equation}
where $\tan\theta = c$. 
Assuming the prescribed projection theorem, choose $E\subset \R^2$ so that, for a.e. $\theta$, 
$$
\mathcal{L}^1(\pi_\theta(E)) = \mathcal{L}^1(  A_\theta)
$$
 and $A_\theta \subset \pi_\theta(E)$. 
Now, combining this with the definition of $A_\theta$ in \eqref{Atheta} and applying \eqref{equiv}, we have
$$\mathcal{H}^1(A_c) 
= n_c \cdot\mathcal{H}^1(A_\theta) 
= n_c\cdot \mathcal{H}^1(\pi_\theta(E)) 
= \mathcal{H}^1( L(E) \cap L_c).$$
An application of Fubini's theorem implies that $\mathcal{L}^2(A) = \mathcal{L}^2(L(E))$. 
\\

To see that Davies' efficient covering theorem implies the prescribed projection theorem, 
let $\{G_\theta\}$ denote a family of measurable subsets of $\R$ indexed by  $\theta \in [0,\pi)$, and consider the set 
$$G= \left\{ (c, y) : y \in \left( \sqrt{1+c^2} \right) G_\theta \text{ and } \tan\theta = c  \right\}.$$
Assuming Davies' efficient covering theorem, let $L(E)$ as in \eqref{line_set} denote a set consisting of lines in the plane so that 
$$\mathcal{L}^2(L(E)) = \mathcal{L}^2(G).$$
It follows again by Fubini that for almost every slice, that is for almost every $c \in \R$, 
$$\mathcal{L}^1(L(E) \cap L_c) = \mathcal{L}^1(G\cap L_c),$$
and so, by \eqref{equiv} and by the definition of $G$, 
$$\mathcal{L}^1(  \pi_\theta(E)) = \mathcal{L}^1(G_\theta).$$
Furthermore, since $L(E)$ covers $G$, then each slice $L(E)\cap L_c$ covers $G\cap L_c$, from which it follows from \eqref{equiv} and the definition of $G$ that $\pi_\theta(E)$ covers $G_\theta$. 
\\

\bibliography{bib}
\bibliographystyle{abbrv}

\end{document}